\documentclass[10pt]{amsart}%{article}
\usepackage{amsmath,amsthm, epsfig, amscd}
\usepackage{psfrag}
\usepackage[psamsfonts]{amssymb}
\usepackage[all]{xy}
\usepackage{mathrsfs}

\numberwithin{equation}{section}

\newtheorem{Thm}{Theorem}[section]
\newtheorem{Prop}[Thm]{Proposition}
\newtheorem{Lem}[Thm]{Lemma}
\newtheorem{Cor}[Thm]{Corollary}

\theoremstyle{definition}
\newtheorem{Rem}[Thm]{Remark}
\newtheorem{Def}[Thm]{Definition}
\newtheorem{Assum}[Thm]{Assumption}
\newtheorem{Prob}[Thm]{Problem}
\newtheorem{Exa}[Thm]{Example}

\newcommand{\mysection}[2]{%
\vspace{2mm}\section{\bf #1}\label{#2}
}

\newcommand{\fig}[1]
        {\raisebox{-0.5\height}
                 {\includegraphics{#1}}
        }

\def\Z{{\mathbb Z}}
\def\R{{\mathbb R}}
\def\Q{{\mathbb Q}}

\def\calA{\mathscr{A}}

\def\calD{\mathscr{D}}
\def\calE{\mathscr{E}}
\def\calF{\mathscr{F}}

\def\calM{\mathscr{M}}

\def\calQ{\mathscr{Q}}

\def\Hom{\mathrm{Hom}}

\def\tcoprod{\textstyle\coprod}

\newcommand{\mapright}[1]{
	\smash{\mathop{
		\hbox to 1cm{\rightarrowfill}}\limits^{#1}}}

\newcommand{\mapleft}[1]{
	\smash{\mathop{
		\hbox to 1cm{\leftarrowfill}}\limits^{#1}}}

\def\Tr{\mathrm{Tr}}

\def\asum#1#2{\sum_{{{#1}\atop{#2}}}}
\def\Conf{\mathrm{Conf}}

\def\bcalM{\overline{\calM}}

\def\bcalD{\overline{\calD}}

\def\ve{\varepsilon}
\def\bConf{\overline{\Conf}}
\def\loc{\mathrm{local}}

\def\acalM{\calM^{\mathrm{AL}}}
\def\bacalM{\bcalM^{\mathrm{AL}}}

\def\wM{\widetilde{M}}
\def\wf{\widetilde{f}}

\def\wcalD{\widetilde{\calD}}
\def\wcalA{\widetilde{\calA}}
\def\wcalF{\widetilde{\calF}}
\def\ind{\mathrm{ind}}
\def\twedge{\textstyle\bigwedge}

\def\irr{\mathrm{irr}}

\def\sign{\mathrm{sign}}
\def\anomaly{\mathrm{anomaly}}
\def\wZ{\widehat{Z}}
\def\Spin{\mathrm{Spin}}

\begin{document}

\title[An invariant of fiberwise Morse functions on surface bundle over $S^1$]{An invariant of fiberwise Morse functions on surface bundle over $S^1$ by counting graphs}
\author{Tadayuki Watanabe}
\address{Department of Mathematics, Shimane University,
1060 Nishikawatsu-cho, Matsue-shi, Shimane 690-8504, Japan}
\email{tadayuki@riko.shimane-u.ac.jp}
\date{\today}
\subjclass[2000]{57M27, 57R57, 58D29, 58E05}

{\noindent\footnotesize {\rm Preprint} (2015)}\par\vspace{15mm}
\maketitle
\vspace{-6mm}
\setcounter{tocdepth}{2}
\begin{abstract}
We apply Lescop's construction of $\Z$-equivariant perturbative invariant of knots and 3-manifolds to the explicit equivariant propagator of ``AL-paths'' given in \cite{Wa2}. We obtain an invariant $\wZ_n$ of certain equivalence classes of fiberwise Morse functions on a 3-manifold fibered over $S^1$, which can be considered as a higher loop analogue of the Lefschetz zeta function and whose construction will be applied to that of finite type invariants of knots in such a 3-manifold. We also give a combinatorial formula for Lescop's equivariant invariant $\calQ$ for 3-manifolds with $H_1=\Z$ fibered over $S^1$. Moreover, surgery formulas of $\wZ_n$ and $\calQ$ for alternating sums of surgeries are given. This gives another proof of Lescop's surgery formula of $\calQ$ for special kind of 3-manifolds and surgeries, which is simple in the sense that the formula is obtained easily by counting certain graphs in a 3-manifold. 
\end{abstract}
\par\vspace{3mm}

%%%%%%%%%%%%%%%%%%%%%%%%%%%%%%
%%%%%%%%%%%%%%%%%%%%%%%%%%%%%%
%%%%%%%%%%%%%%%%%%%%%%%%%%%%%%
\def\baselinestretch{1.07}\small\normalsize

%%%%%%%%%%%%%%%%%%%%%%%%%%%%%%
%%%%%%%%%%%%%%%%%%%%%%%%%%%%%%
\mysection{Introduction}{s:intro}

It is known that the trivial connection contributions of Chern--Simons perturbation theory (\cite{AS, Ko, BC1, BC2, KT} etc.) for homology 3-spheres give very fine topological invariants, which are fine enough to dominate all real-valued finite type invariants. In \cite{Oh1, Oh2}, T.~Ohtsuki defined a $\Z$-equivariant perturbative invariant of 3-manifolds with $b_1=1$ and proved that it is also very fine for 3-manifolds with $b_1=1$. This suggests that the theory of perturbative quantum invariants is quite rich also for 3-manifolds with $b_1>0$. Shortly after Ohtsuki's results appeared, C.~Lescop developed another theory of $\Z$-equivariant perturbative invariant using configuration spaces (\cite{Les2, Les3, Les4}), which can be seen as a $\Z$-equivariant version of Chern--Simons perturbation theory of homology 3-spheres. In \cite{Les2, Les3, Les4}, two kinds of interpretations of $\Z$-equivariant invariant from configuration spaces are given:
\begin{enumerate}
\item As a $\Z$-equivariant invariant of knots in a 3-manifold. This takes values in a space $\calA_n(\widehat{\Lambda})$ of Jacobi diagrams (Definition~\ref{def:jacobi}) whose edges are colored by rational functions in a formal variable $t$. 
\item As an invariant of 3-manifolds with $b_1=1$ by considering the equivalence class of the $\Z$-equivariant invariant of knots with respect to some equivalence relation in the target space.
\end{enumerate}
Lescop proved in \cite{Les4} in a topological method that the invariant of (1) is universal among Garoufalidis--Rozansky's finite type invariants of knots in an integral homology 3-sphere defined by null-claspers (\cite{GR}), which are related to the loop expansion of Kontsevich's universal Vassiliev invariant of knots. Lescop's knot invariant can be considered as a configuration space version of Kricker's rational invariant $Z^\mathrm{rat}$ for knots defined combinatorially (\cite{GK}), which is universal among finite type invariants of \cite{GR}. She also proved in \cite{Les2} that for $M$ with $H_1(M;\Z)=\Z$ and with trivial Alexander polynomial, the invariant of (2) coincides with the 2-loop part of the invariant of Ohtsuki defined in \cite{Oh2} for 3-manifolds with $b_1=1$, up to normalization. In the definitions of Lescop's invariants, a knot $K$ in a 3-manifold $M$ is fixed and a certain 3-cycle $ST(K)$ in the boundary of the equivariant configuration space $\bConf_{K_2}(M)$ of 2 distinct points in $M$ ($\Z$-covering of the configuration space $\bConf_2(M)$ of 2 distinct points in $M$, definition in \S\ref{ss:equiv-conf}), which is associated to $K$, is considered. Then the 3-cycle has an extension to a 4-chain in $\bConf_{K_2}(M)$ with coefficients in the field of rational functions, which is called an equivariant propagator. For example, the term for the $\Theta$-graph in Lescop's equivariant invariant is given by the ``equivariant triple intersection'' in $\bConf_{K_2}(M)$ among equivariant propagators.

In this paper, we apply Lescop's construction to the explicit equivariant propagator of ``AL-paths'' given in \cite{Wa2}. Namely, we introduced in \cite{Wa2} the notion of AL-paths in a 3-manifold $M$ fibered over $S^1$ (definition in \S\ref{ss:AL}) and gave an explicit 4-chain $Q(\xi)$ in $\bConf_{K_2}(M)$ by the moduli space of AL-paths of a fiberwise gradient $\xi$ of a fiberwise Morse function on $M$. Roughly, an AL-path of $\xi$ is a piecewise smooth path consisting of segments each of which is either a part of a critical locus of $\xi$ (vertical segment) or a flow line of $-\xi$ (horizontal segment). Then, for example, the equivariant triple intersection of (parallel copies of) $Q(\xi)$ is given by a generating function of counts of certain graphs in $M$ such that each edge is an AL-path (AL-graphs, Figure~\ref{fig:AL-graph-generic}). 

We shall give two results. First, we take a fiberwise Morse function on $M$ as an extra structure to define an $\calA_n(\widehat{\Lambda})$-valued invariant. We show that our invariant $\wZ_n$ is an invariant of concordance classes of fiberwise Morse functions on $M$ (definition in \S\ref{ss:FMF}) and of spin structures on $M$ (Theorem~\ref{thm:main}), by a bifurcation argument similar to Hatcher--Wagoner \cite{HW}. Roughly, concordance of fiberwise Morse function is analogous to isotopy of closed braid in $M$. Though the concordance relation looks too restricted, it is enough for defining equivariant perturbative isotopy invariants for nullhomologous knots in $M$. We will write in a subsequent paper \cite{Wa3} about the knot invariants which count AL-graphs with univalent vertices attached to a knot. 

Next, concerning (2), we modify our $Q(\xi)$ in order to utilize Lescop's result because unlike Lescop's equivariant propagator in \cite{Les2} the boundary of $Q(\xi)$ does not concentrate on a knot $K$. By adding a bordism in $\partial \bConf_{K_2}(M)$ between $\partial Q(\xi)$ and a multiple of a 2-sphere bundle over $K$, we give a combinatorial formula for Lescop's 3-manifold invariant $\calQ$ in \cite{Les2} for fibered 3-manifolds with $H_1(M)=\Z$ as a generating function of counts of AL-graphs (Theorem~\ref{thm:Q}) with some geometric correction terms. We remark that Lescop's invariant is actually nontrivial, as shown in \cite{Les2}. By using the combinatorial formulas for $\wZ_n$ and $\calQ$, we derive surgery formulas of the invariant $\wZ_n$ (Theorem~\ref{thm:surgery}) and of $\calQ$ (Theorem~\ref{thm:surgery2}). The surgery formula suggests that $\wZ_n$ has a property similar to finite type invariant as does the perturbative invariant for homology 3-spheres (\cite{KT}). This gives another proof of Lescop's surgery formula of $\calQ$ in \cite{Les2} for special kind of 3-manifolds and surgeries, which is simple in the sense that the formula is obtained easily by counting certain graphs in a 3-manifold. 

In computing the values of $\wZ_n$ and $\calQ$ for concrete examples, we show that the homology of the mapping torus of a diffeomorphism of a closed surface is given by counting AL-paths (Proposition~\ref{prop:AL-homology}). The count of AL-paths between fibers of a surface bundle gives the twisted tensor product of K.~Igusa (\cite{Ig2}), which in this setting computes the homology of the mapping torus. 

\subsection{Conventions}
In this paper, manifolds and maps between them are assumed to be smooth unless otherwise noticed. By an $n$-dimensional {\it chain} in a manifold $X$, we mean a {\it finite} linear combination of smooth maps from oriented compact $n$-manifolds with corners to $X$. We understand a chain as a chain of smooth simplices by taking triangulations of manifolds. Let $C_i(X)$ denote the group of piecewise smooth chains in $X$ and let $C^i(X)=\Hom(C_i(X),\Z)$. We represent an orientation of a manifold $X$ by a non-vanishing section of $\twedge^{\dim{X}} T^*X$ and denote it by $o(X)$. We consider a coorientation $o^*_X(V)$ of a submanifold $V$ of a manifold $X$ as an orientation of the normal bundle of $V$ and represent it by a differential form in $\Gamma^\infty(\twedge^\bullet T^*X|_{V})$. We identify the normal bundle $N_V$ with the orthogonal complement $TV^\perp$ in $TX$, by taking a Riemannian metric on $X$. We equip orientation or coorientation of $V$ so that the identity
\[ o(V)\wedge o^*_X(V)\sim o(X) \]
holds, where we say that two orientations $o$ and $o'$ are equivalent ($o\sim o'$) if they are related by multiple of a positive function. $o(V)$ determines $o^*_X(V)$ up to equivalence and vice versa. We orient boundaries of an oriented manifold by the inward normal first convention. For a sequence of submanifolds $A_1,A_2,\ldots,A_r\subset W$ of a smooth Riemannian manifold $W$, we say that the intersection $A_1\cap A_2\cap \cdots\cap A_r$ is {\it transversal} if for each point $x$ in the intersection, the subspace $N_xA_1+N_xA_2+\cdots+N_xA_r\subset T_xW$ spans the direct sum $N_xA_1\oplus N_xA_2\oplus\cdots\oplus N_xA_r$, where $N_xA_i$ is the orthogonal complement of $T_xA_i$ in $T_xW$ with respect to the Riemannian metric. 

Let $\Q[t_1^{\pm 1},\ldots,t_n^{\pm 1}]$ denote the ring of Laurent polynomials in $n$ variables. Let $\Q(t)$ denote the field of fractions of $\Q[t^{\pm 1}]$ and we identify $\Q(t)$ with a subset of the field of formal power series in $t$ with finitely many negative degree terms.

\subsection{Organization}
In \S\ref{s:preliminaries}, we review the definitions of relevant graphs, concordance of fiberwise Morse functions, AL-paths and equivariant propagator. In \S\ref{s:invariants}, we give definitions of two invariants $\wZ_n$ and $\calQ$, which are mainly based on Lescop's construction in \cite{Les2,Les3,Les4}. We claim that one is an invariant of concordance classes of fiberwise Morse functions on a fibered 3-manifold over $S^1$. The other one gives a combinatorial formula for Lescop's invariant for fibered 3-manifolds with $b_1=1$ in \cite{Les2, Les3}. In \S\ref{s:invariance}, we prove that $\wZ_n$ is invariant under a concordance of fiberwise Morse functions by a bifurcation argument. In \S\ref{s:sformula}, we give a surgery formula of $\wZ_n$. Namely, we consider a special kind of Torelli surgery and we give an explicit formula of the value of $\wZ_n$ for an alternating sum of surgeries. In \S\ref{s:homology}, we define a chain complex by counting AL-paths. The homology of the chain complex is naturally isomorphic to the homology of the total space of the fibration. This is used to understand the surgery formula.

%\tableofcontents

%%%%%%%%%%%%%%%%%%%%%%%%%%%%%%
%%%%%%%%%%%%%%%%%%%%%%%%%%%%%%
\mysection{Preliminaries}{s:preliminaries}

%%%%%%%%%%%%%%%%%%%%%%%%%%%%%%
\subsection{Graphs}\label{ss:graph}

By a {\it graph}, we mean a finite connected graph with each edge oriented, i.e., an ordering of the boundary vertices of an edge is fixed. We allow multiple edges and self-loops. A {\it labeled graph} is a graph $\Gamma$ equipped with bijections $\alpha:\{1,2,\ldots,n\}\to V(\Gamma)$ and $\beta:\{1,2,\ldots,\ell\}\to E(\Gamma)$, where $V(\Gamma)$ is the set of vertices of $\Gamma$ and $E(\Gamma)$ is the set of edges of $\Gamma$. We will identify $V(\Gamma)$ and $E(\Gamma)$ with the sets of labels through $\alpha$ and $\beta$ respectively. Let $E^\rho(\Gamma)$ denote the subset of $E(\Gamma)$ consisting of self-loop edges and let $E^I(\Gamma)=E(\Gamma)\setminus E^\rho(\Gamma)$. An {\it orientation} of a graph is an orientation of the real vector space
\[ \R^{V(\Gamma)}\oplus \bigoplus_{e\in E(\Gamma)}\R^{H(e)}, \]
where $H(e)=\{e_+,e_-\}$ is the two-element set of `half-edges', namely $e_-=\varphi^{-1}[0,\frac{1}{2}]$ and $e_+=\varphi^{-1}[\frac{1}{2},1]$ for an orientation preserving homeomorphism $\varphi:e\to [0,1]$. A labeled graph $(\Gamma,\alpha,\beta)$ gives a canonical orientation $o(\Gamma,\alpha,\beta)$. A {\it vertex-orientation} of a vertex $v$ in a graph is a cyclic ordering of the edges incident to $v$. 

Let $R$ be a commutative ring with 1. For a trivalent graph $\Gamma$, an {\it $R$-coloring} of $\Gamma$ is an assignment of an element of $R$ to every edge of $\Gamma$. An $R$-coloring is represented by a map $\phi:E(\Gamma)\to R$. The {\it degree} of a trivalent graph is defined as half the number of vertices. 

\begin{Def}[Garoufalidis--Rozansky \cite{GR}]\label{def:jacobi}
Let $\Lambda=\Q[t^{\pm 1}]$ and $\widehat{\Lambda}=\Q(t)$. Let $\calA_n(\Lambda)$ (resp. $\calA_n(\widehat{\Lambda})$) be the vector space over $\Q$ spanned by pairs $(\Gamma,\phi)$, where $\Gamma$ is a trivalent graph of degree $n$ with vertex-orientation and $\phi$ is a $\Lambda$-coloring (resp. $\widehat{\Lambda}$-coloring) of $\Gamma$, quotiented by the relations AS, IHX, Orientation reversal, Linearity, Holonomy (Figure~\ref{fig:relations}) and automorphisms of oriented graphs. 
\end{Def}
\begin{figure}
\fig{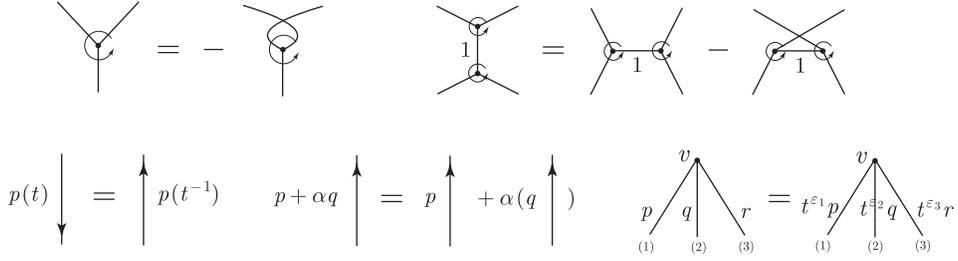}
\caption{The relations AS, IHX, Orientation reversal, Linearity and Holonomy. $p,q,r\in \Lambda$ (or $p,q,r\in \widehat{\Lambda}$), $\alpha\in\Q$. The exponent $\ve_i$ is $1$ if the $i$-th edge is oriented toward $v$ and otherwise $-1$.}\label{fig:relations}
\end{figure}

It is known that there is a canonical correspondence between an orientation of a trivalent graph and a vertex-orientation to each vertex (e.g., \cite{CV}). 

We denote a pair $(\Gamma,\phi)$ by $\Gamma(\phi)$ or by $\Gamma(\phi(e_1),\phi(e_2),\ldots,\phi(e_\ell))$. We say that a $\Lambda$-colored graph $\Gamma(\phi)$ is a {\it monomial} if for each edge $e$ of $\Gamma$, $\phi(e)$ is a power of $t$. In this case, we may consider $\phi$ as a map $E(\Gamma)\to \Z$ by identifying $t^p$ with $p$. There is a bijective correspondence between the equivalence class of a monomial labeled graph $\Gamma(\phi)$ modulo the Holonomy relation and the homotopy class of a continuous map $c:\Gamma\to S^1$, or the cohomology class $[c]\in H^1(\Gamma;\Z)=[\Gamma,S^1]$. 

%\clearpage

%%%%%%%%%%%%%%%%%%%%%%%%%%%%%%
\subsection{Fiberwise Morse functions and their concordances}\label{ss:FMF}

Let $\kappa:M\to S^1$ be a smooth fiber bundle with fiber diffeomorphic to a closed connected oriented 2-manifold $\Sigma$. We equip $M$ with a Riemannian metric. We fix a fiberwise Morse function $f:M\to \R$ and its gradient $\xi:M\to \mathrm{Ker}\,d\kappa$ along the fibers that satisfies the parametrized Morse--Smale condition, i.e., the descending manifold loci and the ascending manifold loci for $\xi$ are mutually transversal in $M$. We consider only fiberwise Morse functions that are {\it oriented}, i.e., the bundles of negative eigenspaces of the Hessians along the fibers on the critical loci are oriented. There always exists an oriented fiberwise Morse function on $M$ (e.g., \cite{Wa2}). The graph of critical values of $f|_{\kappa^{-1}(c)}$ forms a diagram in $\R\times S^1$ (Cerf's graphic), which consists of closed immersed curves. See Figure~\ref{fig:AL-path}. For a generic choice of $f$, the intersection of curves in its graphic consists of transversal double points, which are called a {\it level exchange points}. 

\begin{Def}
A {\it generalized Morse function (GMF)} is a $C^\infty$ function on a manifold with only Morse or birth-death singularities (\cite[Appendix]{Ig1}). A {\it fiberwise GMF} is a $C^\infty$ function $f:M\to \R$ whose restriction $f_c=f|_{\kappa^{-1}(c)}:\kappa^{-1}(c)\to \R$ is a GMF for all $c\in S^1$. A {\it critical locus} of a fiberwise GMF is the subset of $M$ consisting of critical points of $f_c$, $c\in S^1$. A fiberwise GMF is {\it oriented} if it is oriented outside birth-death loci and if birth-death pairs near a birth-death locus have incidence number 1.
\end{Def}

It is known that for a pair of fiberwise Morse functions $f_0,f_1:M\to \R$, there exists a homotopy $\widetilde{f}=\{f_s\}_{s\in [0,1]}$ between $f_0$ and $f_1$ in the space of oriented GMF's on $M$ which gives an oriented fiberwise GMF on the surface bundle $\kappa\times\mathrm{id}:M\times [0,1]\to S^1\times [0,1]$. 

\begin{Def}
We say that the homotopy $\widetilde{f}$ is a {\it concordance} if each birth-death locus of $\wf$ in $M\times [0,1]$ projects to a closed curve that is not nullhomotopic in $S^1\times[0,1]$. 
\end{Def}

\begin{Rem}
There exists a pair $f_0,f_1$ of oriented fiberwise Morse functions as above that are homotopic through a family of oriented GMF's but not concordant because there may be a birth-death locus for a homotopy with nullhomotopic projection in $S^1\times [0,1]$ which can not be removed by deformation. In \cite[\S{1.8}]{Wa2}, we considered a fiberwise gradient of a fiberwise Morse function as a kind of a limit of the nonsingular vector field $\mathrm{grad}\,\kappa$. In this way, a concordance would correspond to a certain isotopy of such a vector field. A birth-death locus in a concordance would correspond to a birth or a cancellation of closed orbits. 
\end{Rem}

%%%%%%%%%%%%%%%%%%%%%%%%%%%%%%
\subsection{AL-paths}\label{ss:AL-path}

Let $\pi:\wM\to M$ be the $\Z$-covering associated to $\kappa$. Let $\overline{\kappa}:\widetilde{M}\to \R$ be the lift of $\kappa$, $\overline{f}:\wM\to \R$ denote the $\Z$-invariant lift of $f$, and $\overline{\xi}$ denote the lift of $\xi$. We say that a piecewise smooth embedding $\sigma:[\mu,\nu]\to M$ is {\it descending} if $\overline{\kappa}(\overline{\sigma}(\mu))\geq \overline{\kappa}(\overline{\sigma}(\nu))$ and $\overline{f}(\sigma(\mu))\geq \overline{f}(\sigma(\nu))$, where $\overline{\sigma}:[\mu,\nu]\to\wM$ is a lift of $\sigma$. We say that $\sigma$ is {\it horizontal} if $\mathrm{Im}\,\sigma$ is included in a single fiber of $\kappa$ and say that $\sigma$ is {\it vertical} if $\mathrm{Im}\,\sigma$ is included in a critical locus of $f$.

\begin{Def}\label{def:al-path}
Let $x,y$ be two points of $M$. An {\it AL-path from $x$ to $y$} is a sequence $\gamma=(\sigma_1,\sigma_2,\ldots,\sigma_n)$, where
\begin{enumerate}
\item for each $i$, $\sigma_i$ is a descending embedding $[\mu_i,\nu_i]\to M$ for some real numbers $\mu_i,\nu_i$ such that $\mu_i< \nu_i$, 
\item for each $i$, $\sigma_i$ is either horizontal or vertical with respect to $f$,
\item if $\sigma_i$ is horizontal, then $\sigma_i$ is a flow line of $-\xi$, possibly broken at critical loci,
\item $\sigma_1(\mu_1)=x$, $\sigma_n(\nu_n)=y$,
\item $\sigma_i(\nu_i)=\sigma_{i+1}(\mu_{i+1})$ for $1\leq i<n$,
\item if $\sigma_i$ is horizontal (resp. vertical) and if $i<n$, then $\sigma_{i+1}$ is vertical (resp. horizontal).
\end{enumerate}
We say that two AL-paths are {\it equivalent} if they differ only by reparametrizations on segments.
\end{Def}
See Figure~\ref{fig:AL-path} for an example of an AL-path. We remark that we do not allow $\sigma_i$ to be a constant map. For an AL-path $\gamma=(\sigma_1,\sigma_2,\ldots,\sigma_n)$, we write
\[ \mathrm{Im}\,\gamma=\bigcup_{i=1}^n \mathrm{Im}\,\sigma_i. \] 
\begin{figure}
\fig{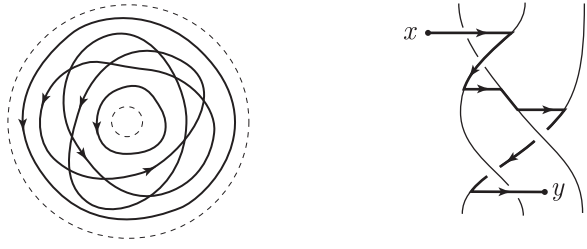}
\caption{Cerf's graphic and an AL-path}\label{fig:AL-path}
\end{figure}
%\clearpage

%%%%%%%%%%%%%%%%%%%%%%%%%%%%%%
\subsection{Fulton--MacPherson compactification of configuration spaces}\label{ss:FM}

We recall the Fulton--MacPherson type compactification of configuration spaces of real oriented manifolds by Kontsevich. See \cite{Ko, BT, Les1} for detail. For a closed $d$-dimensional manifold $M$, the configuration space $\Conf_n(M)$ is the complement of the closed subset
\[ \Sigma=\{(x_1,\ldots,x_n)\in M^n\,;\,\mbox{$x_i=x_j$ for some $i\neq j$}\}\subset M^n. \]
There is a natural filtration $\Sigma=\Sigma_{n-1}\supset\cdots\supset\Sigma_2\supset\Sigma_1$ with
\[ \Sigma_j=\{(x_1,\ldots,x_n)\in M^n\,;\,\#\{x_1,\ldots,x_n\}\leq j\}. \]
The difference $\Sigma_{i+1}-\Sigma_i$ is a disjoint union of smooth submanifolds of $M^n-\Sigma_i$. This property allows one to iterate (real) blow-ups along the filtration from the deepest one: First, one can consider the blow-up $B\ell_{\Sigma_0}(M^n)$ along the $d$-dimensional submanifold $\Sigma_0$ of $M^n$ with oriented normal bundle. Recall that a blow-up replaces a submanifold with its normal sphere bundle. Since the closure of $\Sigma_1-\Sigma_0$ in $B\ell_{\Sigma_0}(M^n)$ is also a disjoint union of smooth submanifolds (with boundaries), one can apply another blow-up along it, and so on. After the blow-ups along all the strata of $\Sigma$ of codimension $\geq 1$, one obtains a smooth compact manifold with corners $\bConf_n(M)$. 

The space $\bConf_n(M)$ has a natural stratification corresponding to bracketings of the $n$ letters $1,2,\ldots,n$, e.g., $((137)(25))46$ (see \cite{FM}). Roughly speaking, a pair of brackets corresponds to a face created by one blow-up. For example, the face corresponding to $((137)(25))46$ is obtained by a sequence of blow-ups corresponding to a sequence $1234567\to (12357)46\to ((137)(25))46$. 

The codimension one (boundary) strata of $\bConf_n(M)$ correspond to bracketings of the form $(\cdots)\cdots$, with only one pair of brackets. The stratum $\partial_{\{1,\ldots,j\}}\bConf_n(M)$ of $\partial\bConf_n(M)$ corresponding to the bracketings $(12\cdots j)j+1\cdots n$ is the face created by the blow-up along the closure of the following submanifold of $M^n$:
\[ \Delta_j=\{(x_1,\ldots,x_n)\in M^n\,;\,x_1=\cdots=x_j,\mbox{otherwise distinct}\} \]
in the result of the previous blow-ups. More precisely, $\partial_{\{1,\ldots,j\}}\bConf_n(M)$ can be naturally identified with the blow-ups of the total space of the normal $S^{(j-1)d-1}$-bundle of $\Delta_j\subset M^n$ along the intersection with the closures of deeper diagonals that correspond to deeper bracketings. The fiber of the normal $S^{(j-1)d-1}$-bundle over a point $(x_j,\ldots,x_n)\in\Delta_j$ is
\[ (\{(0,y_2,\ldots,y_j)\in(\R^d)^j\}-\{0\})/\mbox{(dilation)}\cong S^{(j-1)d-1}, \]
where the coordinate $y_i$ corresponds to $x_i-x_1$ (where it makes sense) via a local framing of $T_{x_j}M$. The stratum $\partial_{\{1,\ldots,j\}}\bConf_n(M)$ is a fiber bundle over $\Delta_j$. We denote the fiber of $\partial_{\{1,\ldots,j\}}\bConf_n(M)$ over a point of $\Delta_j$ by $\Conf_j^\mathrm{local}(\R^d)$. We identify $\Conf_j^\mathrm{local}(\R^d)$ with the subset of $\Conf_j(\R^d)$, as
\[ \Conf_j^\mathrm{local}(\R^d)=\Bigl\{(y_1,\ldots,y_j)\in\Conf_j(\R^d)\,;\, y_1=0, \sum_{\ell=2}^j\|y_\ell\|^2=1\Bigr\}. \]
 We denote by $\bConf_j^\mathrm{local}(\R^d)$ the closure of the image of the inclusion $\Conf_j^\mathrm{local}(\R^d)\hookrightarrow \bConf_j(\R^d)$, which is compact. The base space $\Delta_j$ is naturally diffeomorphic to $\bConf_{n-j+1}(M)$ and we denote by $\mathrm{pr}_j:\Delta_j\to M$ the projection $(x_j,\ldots,x_n)\mapsto x_j$. So $\partial_{\{1,\ldots,j\}}\bConf_n(M)$ has the structure of the pullback of the associated $\bConf_j^\mathrm{local}(\R^d)$-bundle of $TM$ by $\mathrm{pr}_j$:
\[ \xymatrix{
  \partial_{\{1,\ldots,j\}}\bConf_n(M) \ar[r] \ar[d] & P\times_{SO(d)}\bConf_j^\loc(\R^d) \ar[d] \\
  \Delta_j \ar[r]^-{\mathrm{pr}_j} & M
}\]
($P\to M$ is the orthonormal frame bundle of $TM$.) The definition of $\partial_A\bConf_n(M)$ for general subset $A\subset \{1,\ldots,n\}$ corresponding to the bracketing $(A)A^c$ is similar. 

It will turn out that the faces of $\partial\bConf_n(M)$ corresponding to coincidence of two points are special among the codimension one strata of $\bConf_n(M)$. We denote by $\partial^\mathrm{pri}\bConf_n(M)$ (`pri' for principal) the union of the faces corresponding to coincidence of two points and we denote by $\partial^\mathrm{hi}\bConf_n(M)$ (`hi' for hidden) the union of all the faces corresponding to coincidence of at least three points. 

%%%%%%%%%%%%%%%%%%%%%%%%%%%%%%
\subsection{Lescop's equivariant configuration space (\cite{Les3, Les4})}\label{ss:equiv-conf}

Let $\Gamma$ be a labeled graph with $n$ vertices and $m$ edges. By the labeling $\alpha:\{1,2,\ldots,n\}\to V(\Gamma)$, we identify $E(\Gamma)$ with the set of ordered pairs $(i,j)$, $i,j\in \{1,2,\ldots,n\}$. Let $M^\Gamma$ denote the set of tuples
\[ (x_1,x_2,\ldots,x_{n};\{\gamma_{ij}\}_{(i,j)\in E(\Gamma)}), \]
where $x_i\in M$ and $\gamma_{ij}$ is the homotopy class of continuous maps $c_{ij}:[0,1]\to S^1$ relative to the endpoints such that $c_{ij}(0)=\kappa(x_i)$ and $c_{ij}(1)=\kappa(x_j)$. We consider $M^\Gamma$ as a topological space as follows. Let $C^0(\Gamma,S^1)$ be the space of continous maps $\Gamma\to S^1$ equipped with the $C^0$-topology and let $C^\Gamma$ be the space that is the pullback in the following commutative square.
\[ \xymatrix{
  C^\Gamma \ar[r] \ar[d] & C^0(\Gamma,S^1) \ar[d]^-{\nu}\\
  M^{n} \ar[r]^-{\kappa\times\cdots\times\kappa} & (S^1)^{n}
} \]
Here, $\nu:C^0(\Gamma,S^1)\to (S^1)^{n}$ assigns to each $\phi:\Gamma\to S^1$ the images of the $n$ vertices of $\Gamma$ under $\phi$. The fiberwise quotient map $C^\Gamma\to M^\Gamma$ by the homotopy relation of edges gives $M^\Gamma$ the quotient topology. The forgetful map $\overline{\pi}:M^\Gamma\to M^n$ is a locally trivial fibration, which is a $\Z^m$-covering. There is a canonical bijection between the set of connected components in $M^\Gamma$ with $H^1(\Gamma;\Z)=[\Gamma,S^1]\approx \Z^{1-\chi(\Gamma)}$. Note that the covering $M^\Gamma$ may depend on $\kappa$ if $b_1(M)>1$.

Let $\bConf_{\Gamma}(M)$ be the space obtained from $M^\Gamma$ by blowing-up along all the lifts of the diagonals in $M^n$ as in \S\ref{ss:FM}. The forgetful map
\[ \overline{\pi}:\bConf_\Gamma(M)\to \bConf_n(M), \]
is a $\Z^m$-covering. Since $\bConf_{\Gamma}(M)$ is naturally a $\Z^m$-space by the covering translation, the twisted homology
\[ H_i(\bConf_{\Gamma}(M))\otimes_{\Lambda_{\Gamma}}\widehat{\Lambda}_{\Gamma}, \]
where $\Lambda_{\Gamma}=\Q[\{t_{ij}^{\pm 1}\}_{(i,j)\in E(\Gamma)}]$ and $\widehat{\Lambda}_{\Gamma}=\bigotimes_{(i,j)\in E(\Gamma)} \Q(t_{ij})$ (tensor product of $\Q$-modules), is defined. Here, $\Lambda_{\Gamma}$ acts on $\widehat{\Lambda}_{\Gamma}$ by $(\prod_{(i,j)}t_{ij}^{k_{ij}})(\bigotimes_{(i,j)}f_{ij}) = \bigotimes_{(i,j)}t_{ij}^{k_{ij}}f_{ij}$.

\begin{Exa}
If $\Gamma$ is the complete graph $K_2$ with 2 vertices, $\overline{\pi}_{K_2}:\bConf_{K_2}(M)\to \bConf_2(M)$ is a $\Z$-covering. The associated twisted homology is $H_i(\bConf_{K_2}(M))\otimes_{\Lambda}\widehat{\Lambda}$.

If $\Gamma$ is a labeled trivalent graph with $2n$ vertices, we obtain a $\Z^{3n}$-covering $\overline{\pi}_\Gamma:\bConf_{\Gamma}(M)\to \bConf_{2n}(M)$. There is a canonical bijection between the set of connected components in $\bConf_{\Gamma}(M)$ with $H^1(\Gamma;\Z)=\Z^{1+n}$. The associated twisted homology is $H_i(\bConf_{\Gamma}(M))\otimes_{\Lambda_{\Gamma}}\widehat{\Lambda}_{\Gamma}$, where $\Lambda_{\Gamma}=\Q[t_1^{\pm 1},t_2^{\pm 1},\ldots,t_{3n}^{\pm 1}]$ and $\widehat{\Lambda}_{\Gamma}=\Q(t_1)\otimes_\Q\Q(t_2)\otimes_\Q\cdots\otimes_\Q\Q(t_{3n})$.
\end{Exa}
%\clearpage

%%%%%%%%%%%%%%%%%%%%%%%%%%%%%%
\subsection{Equivariant propagator}

Following \cite{Les2, Les3, Les4}, we will consider below an intersection form among chains in the equivariant configuration space $\bConf_\Gamma(M)$. It will be necessary to take a fundamental chain for the intersection called an {\it equivariant propagator}. Here, we take a special one from \cite{Wa2} among equivariant propagators. 

Let $\xi$ be the fiberwise gradient of an oriented fiberwise Morse function on $M$. We say that an AL-path $\gamma$ in $M$ with positive length with respect to the Riemannian metric on $M$ is a {\it closed AL-path} if the endpoints of $\gamma$ coincide. A closed AL-path $\gamma$ gives a piecewise smooth map $\bar{\gamma}:S^1\to M$, which can be considered as a ``closed orbit'' in $M$. We will also call $\bar{\gamma}$ a closed AL-path. A closed AL-path has an orientation that is determined by the orientations of descending and ascending manifolds loci of $\widetilde{\xi}$. See \S\ref{ss:coori} for detail. Then we define the sign $\ve(\gamma)\in\{-1,1\}$ and the period $p(\gamma)$ of $\gamma$ by 
\[ p(\gamma)=|\langle[d\kappa],[\bar{\gamma}]\rangle|,\quad 
\ve(\gamma)=\frac{\langle[d\kappa],[\bar{\gamma}]\rangle}{|\langle[d\kappa],[\bar{\gamma}]\rangle|}.
\]
Let $ST(M)$ be the subbundle of $TM$ of unit tangent vectors. Let $ST(\gamma)$ be the pullback $\bar{\gamma}^*ST(M)$, which can be considered as a piecewise smooth 3-dimensional chain in $\partial\bConf_{K_2}(M)$. We say that two closed AL-paths $\gamma_1$ and $\gamma_2$ are {\it equivalent} if there is a degree 1 homeomorphism $g:S^1\to S^1$ such that $\bar{\gamma}_1\circ g=\bar{\gamma}_2$. The indices of vertical segments in a closed AL-path must be all equal since an AL-path is descending. We define the index $\mathrm{ind}\,\gamma$ of a closed AL-path $\gamma$ to be the index of a vertical segment (critical locus) in $\gamma$, namely, the index of the critical point of $f|_{\kappa^{-1}(c)}$ for any $c\in S^1$ that is the intersection of $\gamma$ with $\kappa^{-1}(c)$. 

Let $M_0=M\setminus\bigcup_{\gamma\,:\,\mathrm{critical\, locus}}\gamma$ and let $s_\xi:M_0\to ST(M_0)$ be the normalization $-\xi/\|\xi\|$ of the section $-\xi$. The closure $\overline{s_\xi(M_0)}$ in $ST(M)$ is a smooth manifold with boundary whose boundary is the disjoint union of circle bundles over the critical loci $\gamma$ of $\xi$, for a similar reason as \cite[Lemma~4.3]{Sh}. The fibers of the circle bundles are equators of the fibers of $ST(\gamma)$. Let $E^-_\gamma$ be the total space of the 2-disk bundle over $\gamma$ whose fibers are the lower hemispheres of the fibers of $ST(\gamma)$ which lie below the tangent spaces of the level surfaces of $\kappa$. Then $\partial\overline{s_\xi(M_0)}=\bigcup_\gamma \partial E_\gamma^-$ as sets. Let
\[ s_\xi^*(M)=\overline{s_\xi(M_0)}\cup \bigcup_\gamma E^-_\gamma\subset ST(M).\]
This is a 3-dimensional piecewise smooth manifold. We orient $s_\xi^*(M)$ by extending the natural orientation $(s_\xi^{-1})^*o(M)$ on $s_\xi(M_0)$ induced from the orientation $o(M)$ of $M$. The piecewise smooth projection $s_\xi^*(M)\to M$ is a homotopy equivalence and $s_\xi^*(M)$ is homotopic to $s_{\hat{\xi}}$. 

Let $K$ be a knot in $M$ such that $\langle[d\kappa],[K]\rangle=-1$. Let $\acalM_{K_2}(\xi)$ be the set of all AL-paths in $M$. There is a natural structure of non-compact manifold with corners on $\acalM_{K_2}(\xi)$. For a closed AL-path $\gamma$, we denote by $\gamma^\irr$ the minimal closed AL-path such that $\gamma$ is equivalent to the iteration $(\gamma^\irr)^k$ for a positive integer $k$ and we call $\gamma^\irr$ the {\it irreducible} factor of $\gamma$. This is unique up to equivalence. If $\gamma=\gamma^\irr$, we say that $\gamma$ is irreducible. We orient $ST(\gamma^\irr)$ so that $[ST(\gamma^\irr)]=p(\gamma^\irr)[ST(K)]$. Note that this may not be the one naturally induced from the orientation of $\gamma^\irr$ but from $\ve(\gamma^\irr)\gamma^\irr$. 

\begin{Thm}[\cite{Wa2}]\label{thm:propagator}
Let $M$ be the mapping torus of an orientation preserving diffeomorphism $\varphi:\Sigma\to \Sigma$ of closed, connected, oriented surface $\Sigma$. Let $\xi$ be the fiberwise gradient of an oriented fiberwise Morse function $f:M\to \R$. 
\begin{enumerate}
\item There is a natural closure $\bacalM_{K_2}(\xi)$ of $\acalM_{K_2}(\xi)$ that has the structure of a countable union of smooth compact manifolds with corners whose codimension 0 strata are disjoint from each other.
\item Let $\bar{b}:\bacalM_{K_2}(\xi)\to M^{K_2}$ be the evaluation map, which assigns the pair of the endpoints of an AL-path $\gamma$ together with the homotopy class of $\kappa\circ \gamma$ relative to the endpoints. Let $B\ell_{\bar{b}^{-1}(\widetilde{\Delta}_M)}(\bacalM_{K_2}(\xi))$ denote the blow-up of $\bacalM_{K_2}(\xi)$ along $\bar{b}^{-1}(\widetilde{\Delta}_M)$. Then $\bar{b}$ induces a map $B\ell_{\bar{b}^{-1}(\widetilde{\Delta}_M)}(\bacalM_{K_2}(\xi))\to \bConf_{K_2}(M)$ and it represents a 4-dimensional $\widehat{\Lambda}$-chain $Q(\xi)$ in $\bConf_{K_2}(M)$ that satisfies the identity
\[ \partial Q(\xi)=s_{\xi}^*(M)+\sum_\gamma (-1)^{\mathrm{ind}\,\gamma}\ve(\gamma)\,t^{p(\gamma)}\,ST(\gamma^\irr), \]
where the sum is taken over equivalence classes of closed AL-paths in $M$. Moreover, $(1-t)^2\Delta(t)Q(\xi)$ is a $\Lambda$-chain, where $\Delta(t)$ is the Alexander polynomial of the fibration $\kappa:M\to S^1$.
\item Suppose that $\kappa$ induces an isomorphism $H_1(M)/\mathrm{Torsion}\cong H_1(S^1)$. Let $K$ be a knot in $M$ such that $\langle[d\kappa],[K]\rangle=-1$. Then the homology class of $\partial Q(\xi)$ in $H_3(\partial\bConf_{K_2}(M))\otimes_\Lambda\widehat{\Lambda}$ is 
\[ [\partial Q(\xi)]=[s_{\xi}^*(M)]+\frac{t\zeta'_{\varphi}}{\zeta_{\varphi}}[ST(K)], \]
where $\zeta_\varphi$ is the Lefschetz zeta function of $\varphi$. 
\end{enumerate}
\end{Thm}

%\clearpage

%%%%%%%%%%%%%%%%%%%%%%%%%%%%%%
\subsection{Coorientation of the strata in $Q(\xi)$}\label{ss:coori}

We recall the orientation convention for $Q(\xi)$ in Theorem~\ref{thm:propagator}. As in \cite{Wa2}, we give the orientations of the strata by coorientations in auxiliary spaces. 

The dimensions of strata in $Q(\xi)$ of AL-paths having vertical segments of index 0 or 2 degenerates in $\bConf_{K_2}(M)$. A nondegenerate (codimension 0) stratum $S$ in $Q(\xi)$ consists of AL-paths $\gamma=(\sigma_1,\sigma_2,\ldots,\sigma_n)$ satisfying one of the following conditions.
\begin{enumerate}
\item $n\geq 2$ and both $\sigma_1$ and $\sigma_n$ are horizontal.
\item $n=1$ and $\sigma_1$ is horizontal. 
\end{enumerate}

In the case (1), $\gamma$ may have several horizontal segments between critical loci of index 1. Each such horizontal segment is the transversal intersection of the descending manifold (locus) $\wcalD_p(\xi)$ of a critical locus $p$ and the ascending manifold (locus) $\wcalA_q(\xi)$ of another critical locus $q$. Such an intersection is called a {\it $1/1$-intersection} in \cite{HW}. A $1/1$-intersection has a sign which is determined by the coorientations of the descending and the ascending manifolds by the identity
\[ o^*_M(\wcalD_p(\xi))_x\wedge o^*_M(\wcalA_q(\xi))_x \sim \ve_x\cdot o(L)_x,
\quad x\in\wcalD_p(\xi)\cap \wcalA_q(\xi), \]
where $L$ is the level surface locus of $f$ including $x$ and $o(L)_x=\iota(-\xi_x)\,o(M)_x$. Let $\ve(\gamma)$ be the product of the signs of $1/1$-intersections on $\gamma$. Then we define the coorientation $o^*_{M^2}(S)_\gamma$ at generic point as
\[ \begin{split}
  o^*_{M^2}(S)_\gamma&=\ve(\gamma)\,o^*_M(\wcalD_p(\xi))_x\wedge o^*_M(\wcalA_q(\xi))_y\\
&\in \twedge^\bullet T^*_{(x,y,[\gamma])}\Conf_{K_2}(M)=\twedge^\bullet T^*_xM\otimes \twedge^\bullet T^*_yM.
\end{split} \]

In the case (2), suppose that $\gamma$ goes from $x\in M$ to $y\in M$. If $T_xM$ is spanned by an orthonormal basis $e_1,e_2,e_3$, then $T_{(x,y)}S$ is spanned by $e_1+Ae_1$, $e_2+Ae_2$, $e_3+Ae_3$, $-\xi_y$, where $A=d\Phi_{-\xi}^s:T_xM\to T_yM$. Let $dx_1,dx_2,dx_3\in T^*_xM$ be the dual basis for $e_1,e_2,e_3$ and let $A_*=(d\Phi_{-\xi}^s)_*:T_x^*M\to T_y^*M$ be the pushforward. Then we define
\[ o(S)_\gamma=(-df)_y\wedge (dx_1+A_*dx_1)\wedge (dx_2+A_*dx_2)\wedge (dx_3+A_*dx_3) \]
and 
\[ o^*_{M^2}(S)_\gamma=*\,o(S)_\gamma\in \twedge^\bullet T^*_xM\otimes \twedge^\bullet T^*_yM, \]
where $*$ is the Hodge star operator. 

%%%%%%%%%%%%%%%%%%%%%%%%%%%%%%
\subsection{Fundamental chain of closed AL-paths}\label{ss:moduli_closed}

We define the 1-cycle
\[ Q'(\xi)=\sum_\gamma \gamma\in C_1(M;\widehat{\Lambda}) \]
in $M$, where the sum is over equivalence classes of all closed AL-paths $\gamma$ for $\xi$ considered as oriented 1-cycles. This is an infinite sum but is well-defined as a $\widehat{\Lambda}$-chain. The orientation of a closed AL-path $\gamma$ is given by $\ve(\gamma)$ times the downward orientation on $\gamma$.

%\clearpage

%%%%%%%%%%%%%%%%%%%%%%%%%%%%%%%%%%%%%%%%
%%%%%%%%%%%%%%%%%%%%%%%%%%%%%%%%%%%%%%%%
\mysection{The invariants $\widehat{Z}_n$ and $\calQ$}{s:invariants}

\subsection{Multilinear form $\langle Q_1,\ldots,Q_{3n}\rangle_\Gamma$ and trace}\label{ss:mlinearform}

Let $\kappa:M\to S^1$ be a smooth fiber bundle with fiber a closed connected oriented 2-manifold. Fix a compact connected oriented 2-submanifold $\Sigma$ of $M$ without boundary such that the oriented bordism class of $\Sigma$ in $M$ corresponds to $[\kappa]$ via the canonical isomorphism $\Omega_2(M)=H_2(M)\cong H^1(M)$. Let $\Gamma$ be a labeled oriented 3-valent graph of degree $n$. 
\begin{enumerate}
\item If $e\in E^I(\Gamma)$, then let $\psi_e:\bConf_{\Gamma}(M)\to \bConf_{K_2}(M)$ denote the projection that gives the endpoints of $e$ together with the associated curve in $S^1$. Take a compact oriented 4-submanifold $F_e$ in $\bConf_{K_2}(M)$ with corners.
\item If $e\in E^\rho(\Gamma)$, then let $\psi_e:\bConf_{\Gamma}(M)\to M$ denote the projection that gives the unique endpoint of $e$ together with the associated curve in $S^1$. Take a compact oriented 1-submanifold $F_e$ in $M$ with boundary.
\end{enumerate}
Note that in both cases $F_i$ is of codimension 2. Then we define
\[ \langle F_1,F_2,\ldots,F_{3n}\rangle_\Gamma
=\bigcap_{e=1}^{3n} \psi_{e}^{-1}(F_e), \]
which gives a compact 0-dimensional submanifold in $\bConf_\Gamma(M)$ if the intersection is transversal. We equip each point $(x_1,x_2,\ldots,x_{2n};\gamma_1,\gamma_2,\ldots,\gamma_{3n})$ of $\langle F_1,F_2,\ldots,F_{3n}\rangle_\Gamma$ with a coorientation (a sign) in $\bConf_{\Gamma}(M)$ by
\[ \bigwedge_{e\in E(\Gamma)} \psi_e^*\,o^*_{\bConf_{K_2}(M)}(F_e)_{(x_e,y_e,\gamma_e)}\in \twedge^\bullet T_{(x_1,\ldots,x_{2n})}M^{2n}. \]
Here, we identify a neighborhood of a point in $\bConf_\Gamma(M)$ with its image of the projection in $M^{2n}$. The coorientation gives a sign as the sign of $\mu$ in the equation $\bigwedge_e\psi_e^*o^*_{\bConf_{K_2}(M)}(F_e)=\mu\, o(M^{2n})$, where $o(M^{2n})$ is the standard orientation of $M^{2n}$. By this, $\langle F_1,F_2,\ldots,F_{3n}\rangle_\Gamma$ represents a 0-chain in $\bConf_{\Gamma}(M)$. This can be extended to tuples of codimension 2 $\Q$-chains by multilinearity. We will denote the homology class of $\langle F_1,F_2,\ldots,F_{3n}\rangle_\Gamma$ (integer) by the same notation.

We extend the form $\langle\cdot,\ldots,\cdot\rangle_\Gamma$ to tuples of codimemsion 2 $\Lambda$-chains $F_k'$ in $\bConf_{K_2}(M)$ or in $M$ as follows. When $k\in E^I(\Gamma)$, suppose that $F_k'$ is of the form $\sum_{\lambda_k=1}^{N_k} \mu_{\lambda_k}^{(k)} \sigma_{\lambda_k}^{(k)}$, where $\mu_{\lambda_k}^{(k)}\in \Lambda$ and $\sigma_{\lambda_k}^{(k)}$ is a compact oriented smooth 4-submanifold in $\bConf_{K_2}(M)[0]$ with corners, where $\bConf_{K_2}(M)[0]$ is the subspace of $\bConf_{K_2}(M)$ consisting of $(x_1,x_2,\gamma_{12})$ such that $\gamma_{12}$ is represented by an arc in $M$ whose algebraic intersection with $\Sigma$ is $0$. When $k\in E^\rho(\Gamma)$, suppose that $F_k'$ is of the form $\sum_{\lambda_k=1}^{N_k} \mu_{\lambda_k}^{(k)} \sigma_{\lambda_k}^{(k)}$, where $\sigma_{\lambda_k}^{(k)}$ is a piecewise smooth path in $M$ transversal to $\Sigma$ and $\mu_{\lambda_k}^{(k)}(t)=\alpha_{\lambda_k}^{(k)}t^{\sigma_{\lambda_k}^{(k)}\cdot \Sigma}$, $\alpha_{\lambda_k}^{(k)}\in \Q$. Then we define
\[ \begin{split}
 \langle F_1',F_2',\ldots,F_{3n}'\rangle_\Gamma
&=\sum_{\lambda_1,\lambda_2,\ldots,\lambda_{3n}}
\mu_{\lambda_1}^{(1)}(t_1)\mu_{\lambda_2}^{(2)}(t_2)\cdots\mu_{\lambda_{3n}}^{(3n)}(t_{3n})
\langle \sigma_{\lambda_1}^{(1)},\sigma_{\lambda_2}^{(2)},\ldots,\sigma_{\lambda_{3n}}^{(3n)}\rangle_\Gamma\\
&\in C_0(\bConf_{\Gamma}(M))\otimes\Q,
\end{split} \]
which can be considered as a 0-chain in $\bConf_{2n}(M)$ with coefficients in $\Q[t_1^{\pm 1},t_2^{\pm 1},\ldots, t_{3n}^{\pm 1}]$. This is multilinear by definition. 

Next, we extend the form $\langle\cdot,\ldots,\cdot\rangle_\Gamma$ to tuples of codimension 2 $\widehat{\Lambda}$-chains in $\bConf_{K_2}(M)$ or $M$ as follows. Let $Q_1,Q_2,\ldots,Q_{3n}$ be codimension 2 $\widehat{\Lambda}$-chains in $\bConf_{K_2}(M)$ or $M$ depending on whether the corresponding edge is not a self-loop or a self-loop. Then there exist $p_1,p_2,\ldots,p_{3n}\in \Lambda$ such that $p_iQ_i$ is a $\Lambda$-chain for each $i$. Then we define
\[ \begin{split}
  \langle Q_1,Q_2,\ldots,Q_{3n}\rangle_\Gamma
&=\langle p_1Q_1,p_2Q_2,\ldots,p_{3n}Q_{3n}\rangle_\Gamma
\,p_1(t_1)^{-1} p_2(t_2)^{-1}\cdots  p_{3n}(t_{3n})^{-1}\\
&\in C_0(\bConf_{\Gamma}(M))\otimes_{\Lambda_\Gamma}\widehat{\Lambda}_\Gamma,
\end{split} \]
which can be considered as a 0-chain in $\bConf_{2n}(M)$ with coefficients in $\Q(t_1)\otimes\cdots\otimes \Q(t_{3n})$. This does not depend on the choices of $p_1,\ldots,p_{3n}$ and this is multilinear by definition. Note that the multilinear form $\langle\cdot,\ldots,\cdot\rangle_\Gamma$ depends on the choice of $\Sigma$. 

We define a $\Q$-linear map $\Tr_\Gamma:\Q(t_1)\otimes\cdots\otimes\Q(t_{3n})\to \calA_n(\widehat{\Lambda})$ by setting
\[ \Tr_\Gamma
\Bigl(\frac{q_1(t_1)}{p_1(t_1)}\otimes\frac{q_2(t_2)}{p_2(t_2)}\otimes\cdots
\otimes \frac{q_{3n}(t_{3n})}{p_{3n}(t_{3n})}\Bigr)
=\Bigl[\Gamma\Bigl(\frac{q_1(t)}{p_1(t)},\frac{q_2(t)}{p_2(t)},\ldots, \frac{q_{3n}(t)}{p_{3n}(t)}\Bigr)\Bigr].\]
This induces a linear map 
\[ \Tr_\Gamma:C_0(\bConf_\Gamma(M))\otimes_{\Lambda_\Gamma}\widehat{\Lambda}_\Gamma \to C_0(\bConf_{2n}(M))\otimes_\Q\calA_n(\widehat{\Lambda}). \]

\begin{Rem}
For the theta-graph $\Theta$ in (\ref{eq:theta}), the definition of $\langle\cdot,\cdot,\cdot\rangle_\Theta$ given above is equivalent to the equivariant triple intersection $\langle\cdot,\cdot,\cdot\rangle_e$ defined in \cite{Les3}. Take three 4-dimensional $\Lambda$-chains $F_X,F_Y,F_Z$ in $\bConf_{K_2}(M)$ and represent as linear combinations of 4-chains in $\bConf_{K_2}(M)[0]$ as follows.
\[ F_X=\sum_{\ell\in\Z}F_X^{(\ell)}t^\ell,\quad
F_Y=\sum_{\ell'\in\Z}F_Y^{(\ell')}t^{\ell'},\quad
F_Z=\sum_{\ell''\in\Z}F_Z^{(\ell'')}t^{\ell''}, \]
where $F_X^{(\ell)},F_Y^{(\ell')},F_Z^{(\ell'')}$ are 4-submanifold of $\bConf_{K_2}(M)[0]$ that are mutually transversal and the sums are finite. Then we have
\[\begin{split}
\langle F_X,F_Y,F_Z\rangle_\Theta &= \sum_{\ell,\ell',\ell''}\bigl\langle F_X^{(\ell)},F_Y^{(\ell')},F_Z^{(\ell'')}\bigr\rangle_{\bConf_{K_2}(M)[0]}\, \Tr_\Theta(t_1^{\ell}\otimes t_2^{\ell'}\otimes t_3^{\ell''})\\
&=\sum_{\ell,\ell',\ell''}\bigl\langle F_X^{(\ell)},F_Y^{(\ell')},F_Z^{(\ell'')}\bigr\rangle_{\bConf_{K_2}(M)[0]}\, \Tr_\Theta(1\otimes t_2^{\ell'-\ell}\otimes t_3^{\ell''-\ell})\\
&=\sum_{\ell,i,j\in\Z}\bigl\langle F_X^{(\ell)},F_Y^{(\ell+i)},F_Z^{(\ell+j)}\bigr\rangle_{\bConf_{K_2}(M)[0]}\, \Tr_\Theta(1\otimes t_2^i\otimes t_3^j)\\
&=\sum_{i,j\in\Z}\bigl\langle F_X,t^{-i}F_Y,t^{-j}F_Z\bigr\rangle_{\bConf_{K_2}(M)}\, \Tr_\Theta(1\otimes t_2^i\otimes t_3^j).\\
\end{split}\]
This agrees with Lescop's equivariant triple intersection in \cite{Les3}. 
\end{Rem}
%%%%%%%%%%%%%%%%%%%%%%%%%%%%%%
\subsection{Invariant $\wZ_n$ of oriented fiberwise Morse functions}\label{ss:def_Z}

\subsubsection{Definition of $Z_n$}
Let $\calE(\kappa)$ denote the set of concordance classes of oriented fiberwise Morse functions for a fibration $\kappa:M\to S^1$. Let $\kappa_1,\kappa_2,\ldots,\kappa_{3n}:M\to S^1$ be fibrations isotopic to $\kappa$. Let $f_i:M\to \R$, $i=1,2,\ldots,3n$, be oriented fiberwise Morse functions for $\kappa_i$ such that $(\kappa_i,f_i)$ is concordant to a pair isotopic to $(\kappa,f)$. Let $\xi_i$ be the fiberwise gradient of $f_i$. Let $Q(\xi_i)$ be the equivariant propagator in Theorem~\ref{thm:propagator} for $\xi_i$ and let $Q'(\xi_i)$ be as in \S\ref{ss:moduli_closed}. Choosing these data generically, we define
\[ \begin{split}
Z_n&=Z_n(\xi_1,\xi_2,\ldots,\xi_{3n})\\
&=\sum_\Gamma \Tr_\Gamma\langle Q^\circ(\xi_1),Q^\circ(\xi_2),\ldots,Q^\circ(\xi_{3n}) \rangle_\Gamma
\in C_0(\bConf_{2n}(M))\otimes_\Q\calA_n(\widehat{\Lambda}),
\end{split} \]
where $Q^\circ(\xi_i)$ in the term for $\Gamma$ is $Q(\xi_i)$ or $Q'(\xi_i)$ depending on whether the corresponding edge in $\Gamma$ is not a self-loop or a self-loop and the sum is over all labeled 3-valent graphs of degree $n$ for all possible edge orientations. We also denote by $Z_n$ the homology class of $Z_n$ in $H_0(\bConf_{2n}(M))\otimes_\Q\calA_n(\widehat{\Lambda})=\calA_n(\widehat{\Lambda})$. 

\subsubsection{Spin structure and canonical framing}
We need a correction term to turn the homology class of $Z_n$ into an invariant under a deformation. Choose a spin structure $\mathfrak{s}$ on $M$. Recall that Rohlin's $\mu$-invariant for $(M,\mathfrak{s})$ is defined by
\[ \mu(M,\mathfrak{s})=\sign\,{W}\,\, (\mbox{mod 16}), \]
where $W$ is a compact spin 4-manifold that spin bounds $(M,\mathfrak{s})$. Kirby--Melvin's $\lambda$-invariant (\cite{KM}) of $(M,\mathfrak{s})$ is given by
\[ \lambda(M,\mathfrak{s})=2(1+r(M))+\mu(M,\mathfrak{s})\,\,\mbox{(mod 4)}, \]
where $r(M)=\mathrm{rank}\,(H_1(M)\otimes \Z_2)$. In \cite{KM}, it is shown that if $W$ is a simply-connected spin 4-manifold that spin bounds $(M,\mathfrak{s})$, for example constructed by attaching 2-handles to $D^4$ along an even framed link, then $\chi(W)\equiv 1+r(M)$ (mod 2). In \cite{KM}, the invariant $\lambda$ was used to determine a canonical stable framing of $TM\oplus \ve^1$ that is compatible with $\mathfrak{s}$. For simplicity, we assume the following condition.
\begin{Assum}\label{assum:mu0}
$\mu(M,\mathfrak{s})=0$ (mod 16).
\end{Assum}

If $\mu(M,\frak{s})\neq 0$ (mod 16), the correction term may still be defined as follows. Consider the 16-fold cyclic covering $M(16)\to M$ which is the pullback of $\kappa$ by the 16-fold covering $S^1\to S^1$. Then $\frak{s}$ induces a natural spin structure $\mathfrak{s}(16)$ on $M(16)$ and we have $\mu(M(16),\mathfrak{s}(16))=0$ (mod 16). Now define the correction term for $(M,\mathfrak{s})$ to be $\frac{1}{16}$ of that for $(M(16),\mathfrak{s}(16))$. This definition is consistent when $\mu(M,\mathfrak{s})=0$ (mod 16).

Under Assumption~\ref{assum:mu0}, one has $\lambda(M,\mathfrak{s})\equiv 0$ or $2$ (mod 4), depending on whether $r(M)$ is odd or even. If $\lambda(M,\mathfrak{s})\equiv 0$ (mod 4), then by \cite[p.97--98]{KM} there is a unique stable framing $\phi$ such that 
\[ \sigma(\phi)=0,\quad d(\phi)=0, \]
where $\sigma$ is the signature defect $\sigma(\phi)=p_1(W;\phi)-3\,\sign\,W$ for $W$ as above and $d(\phi)$ is the degree of the map $M\to S^3$ which gives the direction of the framing $\nu$ of $\ve^1$ considered with respect to the trivialization $\phi$. If $\lambda(M,\mathfrak{s})\equiv 2$ (mod 4), then by \cite[p.97--98]{KM} there is a unique stable framing $\phi$ such that
\[ \sigma(\phi)=0,\quad d(\phi)=1. \]
The latter is the same situation as $\Z$-homology sphere with canonical spin structure discussed in \cite{Wa1}. 

\begin{Lem}\label{lem:cframing}
Under Assumption~\ref{assum:mu0}, we have the following.
\begin{enumerate}
\item If $r(M)$ is odd, then there is a spin 4-manifold $W$ that spin bounds $(M,\mathfrak{s})$ and a unique framing $\varphi$ of $TM$ such that 
\[ \sign\,{W}=0,\quad \chi(W)=0,\quad p_1(W;\varphi\oplus\nu)=0.\]
Hence by \cite[Lemma~2.3]{KM}, the stable framing $\phi=\varphi\oplus\nu$ extends to a framing of $TW$ that is compatible with the spin structure.
\item If $r(M)$ is even, then there is a spin 4-manifold $W$ that spin bounds $(M,\mathfrak{s})$ and a unique stable framing $\phi$ of $TM\oplus\ve^1$ such that 
\[ \sign\,{W}=0,\quad \chi(W)=1,\quad p_1(W;\phi)=0.\]
Hence by \cite[Lemma~2.3]{KM}, the stable framing $\phi$ extends to a framing of $TW$ that is compatible with the spin structure.
\end{enumerate}
\end{Lem}
\begin{proof}
We only give a proof for case (1) since the proof for case (2) is similar. By Assumption~\ref{assum:mu0}, there is a spin 4-manifold $W$ that spin bounds $(M,\mathfrak{s})$ with $\sign\,{W}\equiv 0$ (mod 16). Since $r(M)$ is odd, such a $W$ has even Euler characteristic. Let $K3$ denote the Kummer $K3$ surface and $T^4$ be the 4-torus, both spinnable 4-manifolds. By connected sums of several $K3$ or $-K3$ to $W$, we may assume that $\sign\,W=0$, since $\sign\,{K3}=-16$. Note that for a 4-manifold $X$, one has $\chi(X\# K3)=\chi(X\# (-K3))=\chi(X)+22$, $\chi(X\# T^4)=\chi(X)-2$. Hence we may assume that $\chi(W)=0$ by connect summing several $K3\#(-K3)$ and $T^4$, without changing the signature. 

Now choose the canonical stable framing $\phi$ of $TM\oplus\ve^1$ with $\sigma(\phi)=0$ and $d(\phi)=0$, which is uniquely determined. For the $W$ as above, we have $p_1(W;\phi)=\sigma(\phi)+3\,\sign\,W=0$. By $d(\phi)=0$, there is a framing $\varphi$ of $TM$ such that $\phi$ is homotopic to $\varphi\oplus \nu$. This completes the proof.
\end{proof}

\subsubsection{Anomaly correction term and $\wZ_n$}
Choose $W$ and $\phi$ as in Lemma~\ref{lem:cframing}. When $r(M)$ is odd, one can find a 4-framing of $TW$ and its sub 3-framing $\tau_W^v$ of $TW$ that extends $\varphi$. The 3-framing $\tau_W^v$ spans a rank 3 subbundle $T^vW$ of $TW$. We extend $s_{\xi_i}^*:M\to TM$ to a map $\rho_i:W\to T^vW$, which restricts to a section on $W\setminus \partial W$. We choose $\vec{\rho}_W=(\rho_1,\rho_2,\ldots,\rho_{3n})$ generic as in \cite[\S{2.8.3}]{Wa1} so that we may define
\[ Z_n^\anomaly(\vec{\rho}_W)=\sum_\Gamma \#\calM_\Gamma^\loc(-\vec{\rho}_W)\,[\Gamma(1,1,\ldots,1)]\in\calA_n(\widehat{\Lambda}). \]
See \cite[Definition~2.7]{Wa1} for the definition of $\calM_\Gamma^\loc(-\vec{\rho}_W)$. Roughly, $\calM_\Gamma^\loc(-\vec{\rho}_W)$ is the moduli space of graphs in fibers in the vector bundle $T^vW$ whose $i$-th edge is parallel to $\rho_i$. One can prove that $Z_n^\anomaly(\vec{\rho}_W)$ does not depend on the choices of $W$ as in Lemma~\ref{lem:cframing} and of the extension $\vec{\rho}_W$ (\cite[Proposition~2.12]{Wa1}). Now we define
\[ \widehat{Z}_n=\widehat{Z}_n(\xi_1,\ldots,\xi_{3n},\mathfrak{s})=Z_n(\xi_1,\ldots,\xi_{3n})-Z_n^\anomaly(\vec{\rho}_W)\in\calA_n(\widehat{\Lambda}). \]
When $r(M)$ is even, let $\phi$ be the canonical stable framing of $TM\oplus \ve^1$ as in Lemma~\ref{lem:cframing} (2). This is in a similar situation as the case of $\Z$-homology spheres with canonical spin structure considered in \cite{Wa1}. The stable framing $\phi$ is obtained from the stabilization $\varphi\oplus \nu$ of an honest framing $\varphi$ of $T(M\setminus\infty)$ ($\infty\in M$ : base point) by modifying it on a neighborhood of $\infty$ to a fixed stable framing near $\infty$. Then we may also define the correction term $Z_n^\anomaly(\vec{\rho}_W)$ and $\widehat{Z}_n$ as above. See \cite[\S{2.8.1}]{Wa1} for detail. 
 
\begin{Thm}\label{thm:main}
$\widehat{Z}_n$ is an invariant of $(M,\mathfrak{s},[\kappa],[f])$, where
\begin{enumerate}
\item $\mathfrak{s}$ is a spin structure on $M$,
\item $[\kappa]\in H^1(M)$ is the homotopy class of a fibration $\kappa:M\to S^1$,
\item $[f]\in \calE(\kappa)$ is the concordance class of an oriented fiberwise Morse function $f:M\to \R$.
\end{enumerate}
\end{Thm}

We will denote $Z_n$ and $\wZ_n$ by $Z_n(M,[\kappa],[f])$ and $\widehat{Z}_n(M,\mathfrak{s},[\kappa],[f])$ respectively. Proof of Theorem~\ref{thm:main} is given in \S\ref{s:invariance}. 

According to a result of Laudenbach--Blank \cite{LB}, if the nonsingular closed 1-forms $d\kappa_1$ and $d\kappa_2$ are cohomologous, then they are isotopic. By integrating the 1-parameter family of closed 1-forms, one sees that $\kappa_1$ and $\kappa_2$ are isotopic to each other. 

\begin{Rem}
One can remove the dependence of $\wZ_n$ on spin structures as follows. Recall that for a compact 3-manifold $M$, the number of all spin structures on $M$ is $|H_1(M;\Z_2)|$, which is finite (\cite[Lemma]{Mil}). One may consider the sum
\[ \wZ_n(M,[\kappa],[f])=\frac{1}{|H_1(M;\Z_2)|}\sum_{\mathfrak{s}\in\Spin(M)}\wZ_n(M,\mathfrak{s},[\kappa],[f]) \]
over all spin structures on $M$. 
\end{Rem}

%\clearpage

%%%%%%%%%%%%%%%%%%%%%%%%%%%%%%
\subsection{$Z_n$ and the generating function of counts of AL-graphs}\label{ss:moduli_hori}

Let $f_1,f_2,\ldots,f_{3n}$ and $\xi_1,\xi_2,\ldots,\xi_{3n}$ be as in \S\ref{ss:def_Z}. Let $\Gamma$ be a labeled trivalent graph of degree $n$.
\begin{Def}
Let $\Sigma=\kappa^{-1}(0)$. Suppose that no $1/1$-intersections for $\xi_i$ is on $\Sigma$. We define $\acalM_{\Gamma(\vec{k})}(\Sigma;\xi_1,\xi_2,\ldots,\xi_{3n})$, $\vec{k}=(k_1,k_2,\ldots,k_{3n})$, as the set of piecewise smooth maps $I:\Gamma\to M$ such that
\begin{enumerate}
\item the restriction of $I$ to the $i$-th edge is an AL-path of $\xi_i$,
\item the algebraic intersection of the restriction of $I$ to the $i$-th edge $e_i$ with $\Sigma$ is $k_i$. 
\end{enumerate}
We call such maps {\it AL-graphs} for $(\Sigma;\xi_1,\xi_2,\ldots,\xi_{3n})$ of type $\vec{k}$. We define a topology on $\acalM_{\Gamma(\vec{k})}(\Sigma;\xi_1,\xi_2,\ldots,\xi_{3n})$ as the transversal intersection of smooth submanifolds of $\bConf_\Gamma(M)$, as in \S\ref{ss:mlinearform}. 
\end{Def}
Here the condition (2) implies that $\kappa\circ I:\Gamma\to S^1$ represents the cocycle $\phi(i)=k_i$, $i\in E(\Gamma)$. Let $V\subset \Gamma$ be the union of the preimages of all the vertical segments. Let $H:\overline{\Gamma\setminus V}\to M$ be the restriction of $I$ to the closure of $\Gamma\setminus V$ in $\Gamma$. Then $H$ consists of finitely many ``horizontal'' components each of which consists only of horizontal segments in AL-paths. We say that a component of $H$ is of {\it generic type} if it is a unitrivalent graph whose univalent vertices are mapped by $I$ to critical loci of index 1. 

\begin{figure}
\fig{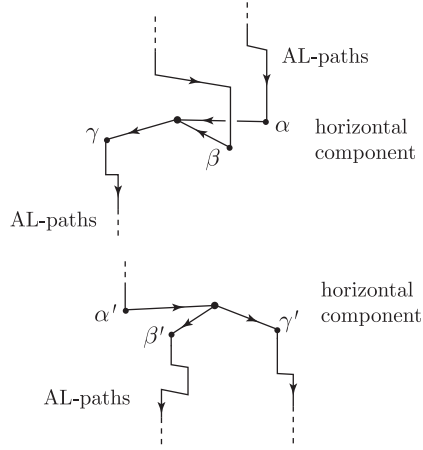}
\caption{An AL-graph of generic type for $\Theta$-graph}\label{fig:AL-graph-generic}
\end{figure}

\begin{Lem}\label{lem:transversality}
For generic choices of $\kappa_1,\kappa_2,\ldots,\kappa_{3n},\xi_1,\xi_2,\ldots,\xi_{3n}$, the moduli space $\acalM_{\Gamma(\vec{k})}(\Sigma;\xi_1,\ldots,\xi_{3n})$ is a compact oriented 0-dimensional manifold for all $\vec{k}$ and for all labeled oriented 3-valent graph $\Gamma$ of degree $n$ simultaneously. 
\end{Lem}
\begin{proof}
By a dimensional reason, we may assume that a horizontal component in $H$ does not have a vertex that is on a critical locus of index 0 or 2, for a generic choice of $(\xi_1,\xi_2,\ldots,\xi_{3n})$. Namely, every horizontal components in $H$ are of generic type. 

The transversality can be proved by an argument similar to \cite[p.~49]{Fu} or \cite[Proposition~2.4]{Wa1} except that the descending (resp. ascending) manifolds of critical points of index 2 (resp. index 1) is replaced with the descending (resp. ascending) manifold loci of critical loci of index 1. 

For the compactness, we use the fact that the subset of $\bacalM_{K_2}(\xi_i)$ of paths of bounded lengths consists of finitely many compact strata (\cite[Lemma~4.1]{Wa2}). Namely, the lift $\pi^{-1}\Sigma=\coprod_{i\in\Z}\Sigma[i]$, $\Sigma[i]=\pi^{-1}(i)$, of $\Sigma$ in $\wM$ decompose $\wM$ into cobordisms: $\wM=\bigcup_{i\in\Z}M[i]$ where $M[i]=\bar{\kappa}^{-1}[i,i+1]$. By the definition of the moduli space of AL-paths given in \cite{Wa2}, the set of AL-paths from a point of $M[k]$ to $M[k-n]$ forms a finite union of finite coverings over compact subsets of $M[k]\times M[k-n]$, where the number of sheets in the covering is the number of AL-paths from an output point in $\Sigma[i]$ to an input point in $\Sigma[k-n+1]$. Hence the moduli space of AL-graphs of given type $\vec{k}$ is compact. 

For a fixed $n$ and for a generic choice of $(\xi_1,\ldots,\xi_{3n})$, there are finitely many possiblities for horizontal components that may be a horizontal part of an AL-graph of degree $n$. Thus the transversality for all the horizontal components implies that $\acalM_{\Gamma(\vec{k})}(\Sigma;\xi_1,\ldots,\xi_{3n})$ is a compact oriented 0-dimensional manifold for all $\vec{k}$ and for all $\Gamma$ of degree $n$. 
\end{proof}

We may identify a point of $\acalM_{\Gamma(\vec{k})}(\Sigma;\xi_1,\ldots,\xi_{3n})$ with an oriented 0-manifold in $\bConf_{\Gamma}(M)$. Hence the moduli space $\acalM_{\Gamma(\vec{k})}(\Sigma;\xi_1,\ldots,\xi_{3n})$ can be counted with signs. The sum of the signs agrees with the sum of coefficients of the terms of $t_1^{k_1}t_2^{k_2}\cdots t_{3n}^{k_{3n}}$ in the power series expansion of $\langle Q^\circ(\xi_1),Q^\circ(\xi_2),\ldots,Q^\circ(\xi_{3n})\rangle_\Gamma$. We denote the sum of signs by $\#\acalM_{\Gamma(\vec{k})}(\Sigma;\xi_1,\ldots,\xi_{3n})$. 

Lemma~\ref{lem:transversality} implies that for generic choices of $\Sigma,\kappa_1,\ldots,\kappa_{3n},\xi_1,\ldots,\xi_{3n}$, an AL-graph $I\in \acalM_{\Gamma(\vec{k})}(\Sigma;\xi_1,\ldots,\xi_{3n})$ consists of finitely many horizontal components of generic type and some AL-paths connecting the univalent vertices of the horizontal components. We say that such an AL-graph is of {\it generic type}. The following proposition follows from Theorem~\ref{thm:propagator} and from definition of $Z_n$. 

\begin{Prop}\label{prop:F} Let $\xi_1,\xi_2,\ldots,\xi_{3n}$ be as in Lemma~\ref{lem:transversality}. For a labeled trivalent graph $\Gamma$, let $F_\Gamma(\Sigma;\xi_1,\xi_2,\ldots,\xi_{3n})$ be the generating function
\[ \sum_{\vec{k}=(k_1,\ldots,k_{3n})\in\Z^{3n}}\#\acalM_{\Gamma(\vec{k})}(\Sigma;\xi_1,\xi_2,\ldots,\xi_{3n})\,t_1^{k_1}t_2^{k_2}\cdots t_{3n}^{k_{3n}}, \]
where $\#\acalM_{\Gamma(\vec{k})}(\Sigma;\xi_1,\xi_2,\ldots,\xi_{3n})$ is the count of AL-graphs of type $\vec{k}$ of generic type. Then there exist Laurent polynomials $P_i^{(\mu)}(t)\in \Lambda$, $i=1,2,\ldots,3n$, $\mu=1,2,\ldots,N$, such that
\[ F_\Gamma(\Sigma;\xi_1,\xi_2,\ldots,\xi_{3n})=\sum_{\mu=1}^N\prod_{i=1}^{3n}\frac{P_i^{(\mu)}(t_i)}{(1-t_i)^2\Delta(t_i)}.\]
Considering this as an element of $\Q(t_1)\otimes_\Q\Q(t_2)\otimes_\Q\cdots\otimes_\Q\Q(t_{3n})$, we have
\[ Z_n(\xi_1,\xi_2,\ldots,\xi_{3n}) =\sum_\Gamma \Tr_\Gamma F_\Gamma(\Sigma;\xi_1,\xi_2,\ldots,\xi_{3n}). \]
\end{Prop}
\begin{Rem}
The trace $\Tr_\Gamma$ can not be directly applied to formal power series. For example, consider the formal power series 
\[\sum_{k\geq 1}t_1^{k}t_2^{k}t_3^{k}.\]
When $\Gamma=\Theta$ in (\ref{eq:theta}) in \S\ref{ss:lescop-inv} below, applying $\Tr_\Theta$ to each monomial $t_1^kt_2^kt_3^k$ corresponds to taking the equivalence class in $\Q[t_1^{\pm 1},t_2^{\pm 1},t_3^{\pm 1}]/(t_1t_2t_3-1)$ by the Holonomy relation. However, the infinite sum of the equivalence classes is not well-defined since $\sum_{k\geq 1} t_1^kt_2^kt_3^k=\sum_{k\geq 1}1=\infty$ if $t_1t_2t_3=1$.
\end{Rem}

%%%%%%%%%%%%%%%%%%%%%%%%%%%%%%
\subsection{A combinatorial formula for Lescop's 2-loop invariant for fibered 3-manifolds with $H_1=\Z$}\label{ss:lescop-inv}

Here, we assume that $H_1(M;\Z)=\Z$. Let $\tau:TM\to \R^3\times M$ be a trivialization of $TM$ and let $s_\tau:M\to ST(M)$ be a section induced by $\tau$ that sends $M$ to $\{v\}\times M$ for a fixed $v\in S^2$. We denote by $s_\tau(M;v)$ the image of $s_\tau$. Suppose that $s_\tau|_K$ agrees with the unit tangent vectors of $TK$. Moreover, by choosing $f$ suitably, we may assume that $K$ agrees with a critical locus of index 0. Let $\check{A}(K):S^1\times [0,1]\to \bConf_2(M)$ be the map defined by $\check{A}(K)(t,u)=(K(t),K(t+u))$ and let $A(K):S^1\times [0,1]\to \bConf_{K_2}(M)$ be the lift of $\check{A}(K)$ such that $S^1\times\{0\}$ is taken to $\bConf_{K_2}(M)[0]$. 

Originally, an equivariant propagator is defined in \cite{Les2} as a 4-dimensional $\widehat{\Lambda}$-chain $Q$ in $\bConf_{K_2}(M)$ such that the chain level identity
\begin{equation}\label{eq:dQ_tau}
 \partial Q=s_\tau(M;v)+I_\Delta(t)\,ST(K')
\end{equation}
holds, where $K'$ is a parallel of $K$ and 
\[ I_\Delta(t)=\frac{1+t}{1-t}+\frac{t\Delta'(t)}{\Delta(t)}, \]
$\Delta(t)$ is the Alexander polynomial of $M$ normalized so that $\Delta(1)=1$ and $\Delta(t^{-1})=\Delta(t)$, and such that
\[ \langle Q,A(K)\rangle_\Z=\sum_{i\in \Z}\langle t^{-i}Q,A(K)\rangle_{\bConf_{K_2}(M)}\,t^i=0. \]
Lescop considered in \cite{Les2} such a $Q$ to define an equivariant invariant of $M$. Since our equivariant propagator $Q(\xi)$ does not satisfy the identity (\ref{eq:dQ_tau}) (compare to Theorem~\ref{thm:propagator}), we extend it by adding a bordism in $ST(M)$ in order to use results of \cite{Les2}. 

\subsubsection{Bordism between closed AL-paths and a knot}
For a closed AL-path $\gamma:S^1\to M$, we shall take a 2-dimensional bordism $V_\gamma$ such that $\partial V_\gamma=\pm\gamma-\mu\,K'$ for an integer $\mu$, as follows. Let $\Sigma=\kappa^{-1}(c)$ for a generic value $c\in S^1$ and let $x_1,x_2,\ldots,x_r\in\Sigma$ be all the intersection points of $\Sigma$ with critical loci of an oriented fiberwise Morse function $f$. Suppose that $K'$ intersects $\Sigma$ transversally at $x_0$ with the intersection number $1$. For each $j$, choose a path $c_j$ on $\Sigma$ from $x_0$ to $x_j$. Let $P_{ij}$ be the set of AL-paths from $x_i$ to $x_j$ which do not intersect $\Sigma$ except the endpoints. Then for $\omega\in P_{ij}$, the 1-cycle $\underline{\omega}=c_i+\ve(\omega)\omega-c_j$ is bordant to $K'$ since $H_1(M)=\Z$. Take a 2-dimensional bordism $V_{\omega}$ such that $\partial V_{\omega}=\underline{\omega}-K'$ and such that $\langle V_\omega,K\rangle =0$, and let 
\[ V_{ij}=\sum_{\omega\in P_{ij}} V_\omega,\quad N(x_i,x_j)=\sum_{\omega\in P_{ij}} \ve(\omega). \]

A closed AL-path $\gamma:S^1\to M$ is cut by $\Sigma$ into segments: $\gamma=\omega_1 \omega_2\cdots \omega_k$, $\omega_\ell\in \bigcup_{i,j}P_{ij}$. Then the 2-dimensional bordism 
\begin{equation}\label{eq:Vsum}
 V_\gamma=\sum_{j=1}^{k}V_{\omega_j}
\end{equation}
satisfies $\partial V_\gamma=\ve(\gamma)\gamma - p(\gamma)K'$.
\begin{Lem}\label{lem:ST(V)} Let $\widehat{V}_\xi$ be the 2-dimensional $\widehat{\Lambda}$-chain
\[ \widehat{V}_\xi=\sum_{1\leq i,j\leq r}(1-tA)_{ji}^{-1}V_{ij}, \]
where $A$ is the matrix $(N(x_i,x_j))$ and $(1-tA)_{ji}^{-1}$ is the $(j,i)$-th entry of $(1-tA)^{-1}$. 
Then we have
\[ 
  \partial ST(\widehat{V}_\xi)=\sum_\gamma(-1)^{\mathrm{ind}\gamma}\ve(\gamma)\,t^{p(\gamma)}ST(\gamma^\irr)-\frac{t\zeta'}{\zeta}ST(K'),
 \]
where the sum is over equivalence classes of all closed AL-paths for $\xi$, and $\zeta$ is the Lefschetz zeta function of the monodromy of the fibration $\kappa$.
\end{Lem}
\begin{proof}
By an argument similar to \cite[Lemma~4.2]{Wa2}, we have
\[   \widehat{V}_\xi=\sum_\gamma(-1)^{\mathrm{ind}\gamma}\ve(\gamma)\,t^{p(\gamma)}\,V_{\gamma^\irr}. \]
Indeed, by decomposing $V_{\gamma^\irr}$ as (\ref{eq:Vsum}), the right hand side can be rewritten as a linear combination of the chains $V_\omega$ where $\omega\in \bigcup_{i,j}P_{ij}$. The coefficient of $V_\omega$ in the right hand side is a power series in $t$, which is the generating function of the numbers of irreducible closed AL-paths that starts from $\omega$, namely, base pointed closed AL-paths. If $\omega\in P_{ij}$, the coefficient of $V_\omega$ is
\[ \sum_{n=1}^\infty (tA)_{ji}^{n-1}=(1-tA)_{ji}^{-1}. \] 

By definition of $V_{\gamma}$, we have
\[\begin{split}
\partial ST(\widehat{V}_\xi)&=\sum_\gamma(-1)^{\ind{\gamma}} \ve(\gamma)\,t^{p(\gamma)}\,(ST(\gamma^\irr)-p(\gamma^\irr)ST(K'))\\
&=\sum_\gamma(-1)^{\ind{\gamma}} \ve(\gamma)\,t^{p(\gamma)}\,ST(\gamma^\irr)-\frac{t\zeta'}{\zeta}ST(K'),
\end{split}\]
where the last equality follows from \cite[Proposition~4.9]{Wa2}.
\end{proof}

\subsubsection{Bordism for the difference of sections}
The following lemma follows from \cite[Proposition~2.12, 4.5]{Les2}. 
\begin{Lem}\label{lem:U}
There exists a 4-dimensional chain $U_\xi$ in $ST(M)$ such that
\[ \partial U_\xi=s_\xi^*(M)-s_\tau(M;v)+(g-1)ST(K'), \]
where $g$ is the genus of $\Sigma$. 
\end{Lem}
\begin{Rem}\label{rem:U}
The 4-chain $U_\xi$ can be taken more explicitly as follows. Let $\hat{\xi}=\xi+\rho\,\mathrm{grad}\,\kappa$, where $\rho$ is a nonnegative smooth function supported on a small tubular neighborhood of the union of all critical loci. Then $\hat{\xi}$ is a nonsingular vector field on $M$ and the image of the section $s_{\hat{\xi}}=-\hat{\xi}/\|\hat{\xi}\|:M\to ST(M)$ can be arbitrarily close to $s_\xi^*(M)$ with respect to the Hausdorff distance. Now one can take $\tau$ so that the restriction of $\tau^{-1}(\{v\}\times M)$ to the complement of a small tubular neighborhood $N$ of $K'\cup\bigcup_{\gamma:\mathrm{critical\, locus}} \gamma$ is parallel to $s_{\hat{\xi}}$ at every point $x\in M\setminus N$. Then $U_\xi$ can be taken as follows.
\begin{enumerate}
\item The restriction of $U_\xi$ to each fiber $ST(x)$ of $ST(M\setminus (K'\cup \bigcup_{\gamma:\mathrm{critical\, locus}}\gamma))$ is the minimal geodesic between $s_\xi^*(M)\cap ST(x)$ and $s_{\hat{\xi}}(x)$, which is a constant path for $x\in M\setminus N$. 
\item $\tau$ may be perturbed so that $s_\tau(M;v)$ is arbitrarily close to $s_\xi^*(M)$ in $M\setminus K'$ although they are singularly different on a small neighborhood of $K'$. The limit of the perturbation of $s_\tau(M;v)$ restricts on $ST(K')$ to a chain that is homologous to $(g-1)ST(K')$ in $ST(K')$. Namely, the limit of $s_\tau(M;v)$ is homologous to $s_\xi^*(M)+(g-1)ST(K')$ and the difference of the restrictions on $ST(K')$ bound a 4-chain in $ST(K')$. We may assume that in a small neighborhood of $ST(K')$ in $ST(M)$ the restriction of $U_\xi$ is given by this 4-chain. 
\end{enumerate}
\end{Rem}

\subsubsection{A formula for Lescop's invariant}
We put
\[ \widehat{Q}(\xi)=Q(\xi)-U_\xi-ST(\widehat{V}_\xi). \]
Let $\Theta$ denote the following labeled oriented graph:
\begin{equation}\label{eq:theta}
 \fig{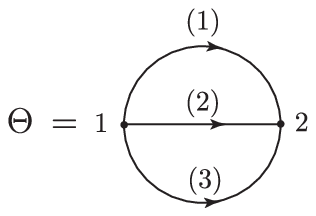} 
\end{equation}
Let $\delta(t)$ be the minimal polynomial of $t:H_1(\wM;\Q)\to H_1(\wM;\Q)$ normalized so that $\delta(1)=1$ and $\delta(t^{-1})=\delta(t)$. Let $O_\delta$ be the subspace of $\calA_1(\widehat{\Lambda})$ spanned by the elements
\[
\Tr_\Theta\Bigl(\frac{t_1^k-t_1^{-k}}{\delta(t_1)}\otimes I_\Delta(t_2)\otimes 1\Bigr)
\]
for all positive integers $k$. 

Let $f_i:M\to \R$, $i=1,2,3$, be oriented fiberwise Morse functions for the fibration $\kappa:M\to S^1$ and let $\xi_i$ be the fiberwise gradient of $f_i$. Let $N_K$ be a small closed tubular neighborhood of $K$ in $M$ and let $K_1,K_2,K_3\subset N_K$ be three parallels of $K$ taken with respect to the trivialization $\tau|_K$. We may assume without loss of generality that $K$ is disjoint from all the closed AL-paths of $\xi_i$. Let $v_1,v_2,v_3\in S^2$ be three points that are close to the fixed point $v\in S^2$. Then by replacing $\xi,K,v$ with $\xi_i,K_i,v_i$ in the definition of $\widehat{Q}(\xi)$ above, we obtain three 4-dimensional chains $\widehat{Q}(\xi_i)$, $i=1,2,3$.

\begin{Thm}\label{thm:Q} We have
\[ \partial \widehat{Q}(\xi_i)=
s_{\tau}(M;v_i)+I_\Delta(t_i)\,ST(K_i), \]
\[ \langle Q(\xi_i),A(K)\rangle_\Z=0\quad (i=1,2,3). \]
After perturbations of $ST(\widehat{V}_{\xi_i})$ and $U_{\xi_i}$ along small inward normal vector fields on $\partial\bConf_{K_2}(M)$ fixing on $\partial ST(\widehat{V}_{\xi_i})$ and $\partial U_{\xi_i}$ for each $i$, we may arrange that the expression 
\[ \calQ(\xi_1,\xi_2,\xi_3)=\Tr_\Theta\langle \widehat{Q}(\xi_1),\widehat{Q}(\xi_2),\widehat{Q}(\xi_3)\rangle_\Theta \]
is well-defined and the class of $\calQ(\xi_1,\xi_2,\xi_3)$ in $\calA_1(\widehat{\Lambda})/O_\delta$ is an invariant of $M$ and the homotopy class of $\tau$. (Up to normalization, $\calQ$ agrees with Lescop's invariant in \cite{Les2} without framing correction term.) For generic choices of $\kappa_i,\xi_i$ and for $U_{\xi_i}$ chosen as in Remark~\ref{rem:U}, we have
\[ \begin{split}
  \calQ(\xi_1,\xi_2,\xi_3)&=\Tr_\Theta\Bigl[\langle Q(\xi_1),Q(\xi_2),Q(\xi_3)\rangle_\Theta
+\langle U_{\xi_1},U_{\xi_2},U_{\xi_3}\rangle_\Theta\\
&+\langle Q(\xi_1),U_{\xi_2},U_{\xi_3}\rangle_\Theta
+\langle U_{\xi_1},Q(\xi_2),U_{\xi_3}\rangle_\Theta
+\langle U_{\xi_1},U_{\xi_2},Q(\xi_3)\rangle_\Theta\\
&+\langle ST(\widehat{V}_{\xi_1}),U_{\xi_2},U_{\xi_3}\rangle_\Theta
+\langle U_{\xi_1},ST(\widehat{V}_{\xi_2}),U_{\xi_3}\rangle_\Theta
+\langle U_{\xi_1},U_{\xi_2},ST(\widehat{V}_{\xi_3})\rangle_\Theta\Bigr]. 
\end{split}\]
\end{Thm}
\begin{proof}
The Alexander polynomial of a fibered 3-manifold $M$ over $S^1$ is of the form $\Delta(t)=c\,t^{-g}\det(1-t\varphi_{*1})$, where $c\in \Q$ and $\varphi_{*i}:H_i(\Sigma;\Q)\to H_i(\Sigma;\Q)$ is the monodromy action of the standard generator of $\pi_1(S^1)$. This together with the formula $\zeta=\zeta_\varphi(t)=\prod_{i=0}^2\det(1-t\varphi_{*i})^{(-1)^{i+1}}$ gives the following identity of the logarithmic derivatives
\[ \frac{t\zeta'}{\zeta}-\frac{t\Delta'(t)}{\Delta(t)}=\frac{2t}{1-t}+g=\frac{1+t}{1-t}+g-1. \]
Hence by Lemma~\ref{lem:U} we have
\[ \begin{split}
  \partial\widehat{Q}(\xi_i)&=\partial Q(\xi_i) - \partial U_{\xi_i}-\partial ST(\widehat{V}_{\xi_i})\\
  &=s_\tau(M;v_i)+\left(\frac{t_i\zeta'(t_i)}{\zeta(t_i)}-(g-1)\right)ST(K_i)=s_\tau(M;v_i)+I_\Delta(t_i)\,ST(K_i).
\end{split} \]
The vanishing of $\langle \widehat{Q}(\xi_i),A(K)\rangle_\Z$ is immediate from Remark~\ref{rem:U}, namely, since the restriction of $\xi_i$ on $K$ is horizontal everywhere, $U_{\xi_i}$ does not meet $A(K)$. Since we assumed that $K$ is a critical locus of $f$ of index 0, $K$ is induded in an ascending manifold locus of $\xi_i$ of index 0, $K$ is transversal to fibers of $\kappa_i$ and $\langle Q(\xi_i),A(K)\rangle_\Z=0$. Moreover, $\langle V_\omega,K\rangle=0$ implies that $\langle ST(\widehat{V}_{\xi_i}),A(K)\rangle_\Z=0$. This shows that $\calQ$ agrees with Lescop's invariant of 3-manifolds in \cite{Les2} without framing correction term. 

Let $C_i\subset M$ be the union of the images of all the closed AL-paths of $\xi_i$ and $K_i$. By a dimensional reason, we may assume that $C_i\cap C_j=\emptyset$ if $i\neq j$. This proves the vanishing of the terms like
\[ \begin{array}{ll}
  \langle Q(\xi_1),Q(\xi_2),U_{\xi_3}\rangle_\Theta,\quad
  &\langle Q(\xi_1),Q(\xi_2),ST(\widehat{V}_{\xi_3})\rangle_\Theta,\\
  \langle Q(\xi_1),U_{\xi_2},ST(\widehat{V}_{\xi_3})\rangle_\Theta.\phantom{\displaystyle{\int}}\\
\end{array}
\]
Moreover, since $\langle \widehat{V}_{\xi_i},\widehat{V}_{\xi_j}\rangle$ is a 1-chain, we may also assume that $\langle \widehat{V}_{\xi_i},\widehat{V}_{\xi_j},C_k\rangle$ is nullhomologous if $i,j,k$ are distinct. This implies the vanishing of the terms like
\[ \langle Q(\xi_1),ST(\widehat{V}_{\xi_2}),ST(\widehat{V}_{\xi_3})\rangle_\Theta,\quad
\langle U_{\xi_1},ST(\widehat{V}_{\xi_2}),ST(\widehat{V}_{\xi_3})\rangle_\Theta. \]

Before the perturbations of $ST(\widehat{V}_{\xi_i})$, the triple intersection 
\begin{equation}\label{eq:ST3}
 \langle ST(\widehat{V}_{\xi_1}),ST(\widehat{V}_{\xi_2}),ST(\widehat{V}_{\xi_3})\rangle_\Theta
\end{equation}
is $ST$ of $\langle \widehat{V}_{\xi_1},\widehat{V}_{\xi_2},\widehat{V}_{\xi_3}\rangle_\Theta$ as a set. But the intersection is not transversal and perturbations of $ST(\widehat{V}_{\xi_i})$ are necessary. At each intersection point $x$ in the triple intersection $\langle \widehat{V}_{\xi_1},\widehat{V}_{\xi_2},\widehat{V}_{\xi_3}\rangle_\Theta$, let $U$ be a small neighborhood of $x$ in $M$. There is a local coordinate $(y,\theta,r)\in U\times S^2\times [0,a)$, $a>0$ small, on a small neighborhood of $ST(U)$ in $\bConf_2(M)$, where $ST(U)$ corresponds to $U\times S^2\times \{0\}$. Then perturb $ST(\widehat{V}_{\xi_i})$ to the level $U\times S^2\times \{c_i\}$ for a small positive number $c_i<a$ by using a cloche function supported on $U$. If $c_1,c_2,c_3$ are mutually distinct, the result of the perturbation has empty intersection in $U\times S^2\times [0,a)$. This shows the vanishing of the term (\ref{eq:ST3}). Hence we obtain the formula of the statement.
\end{proof}

\begin{Rem}
By (an analogue of) Proposition~\ref{prop:F} the first term of the sum formula in Theorem~\ref{thm:Q} counts AL-graphs of generic type. The rest in the formula of Theorem~\ref{thm:Q} have geometric descriptions as follows.
\end{Rem}

\begin{Prop}\label{prop:formula-linking}
For generic choices of $\kappa_i,\xi_i$ and for $U_{\xi_i}$ chosen as in Remark~\ref{rem:U} and Theorem~\ref{thm:Q}, we have the following.
\begin{enumerate}
\item $\langle Q(\xi_1),U_{\xi_2},U_{\xi_3}\rangle_\Theta=\displaystyle\sum_{\gamma}(-1)^{\mathrm{ind}\,\gamma}\ve(\gamma)\,t_1^{p(\gamma)}\ell k_{\gamma^\irr}(\xi_2,\xi_3)\in\Q(t_1)$, where the sum is taken over closed AL-paths for $\xi_1$ and $\ell k_{\gamma^\irr}(\xi_2,\xi_3)$ is the linking number of two parallels of $\gamma^\irr$ given by $\xi_2|_{\gamma^\irr}$ and $\xi_3|_{\gamma^\irr}$ defined with respect to $\tau$. 

\item $\langle ST(\widehat{V}_{\xi_1}),U_{\xi_2},U_{\xi_3}\rangle_\Theta=\langle \widehat{V}_{\xi_1},C(\xi_2,\xi_3)\rangle\in \Q(t_1)$, where $C(\xi_2,\xi_3)$ is the piecewise smooth link in $M$ defined as the projection of $s_{\xi_2}^*(M)\cap s_{\xi_3}^*(M)$, with the orientation determined by the intersection in $ST(M)$.

\item $\langle U_{\xi_1},U_{\xi_2},U_{\xi_3}\rangle_\Theta=\ell k_{C(\xi_2,\xi_3)}(\xi_1)\in \Z$, where $\ell k_{C(\xi_2,\xi_3)}(\xi_1)$ is the linking number of $C(\xi_2,\xi_3)$ and its parallel given by $\xi_1$. 
\end{enumerate}
\end{Prop}
Since Proposition~\ref{prop:formula-linking} is not necessary in the rest of this paper, we omit the proof of Proposition~\ref{prop:formula-linking}. 
%\clearpage

%%%%%%%%%%%%%%%%%%%%%%%%%%%%%%
%%%%%%%%%%%%%%%%%%%%%%%%%%%%%%
\mysection{Proof of invariance of $\wZ_n$}{s:invariance}

%%%%%%%%%%%%%%%%%%%%%%%%%%%%%%
\subsection{Independence of $\Sigma$}\label{ss:indep_sigma}

\begin{Lem}\label{lem:indep_sigma}
The term $\Tr_\Gamma\langle Q^\circ(\xi_1),\ldots,Q^\circ(\xi_{3n})\rangle_\Gamma$ of $Z_n$ does not depend on the choice of $\Sigma$ within its oriented bordism class.
\end{Lem}
\begin{proof}
There are finitely many possibilities for horizontal components in AL-graphs for $(\xi_1,\ldots,\xi_{3n})$. A bordism of $\Sigma$ is decomposed into a sequence of homotopies and attachings of 1-handles that are disjoint from all the horizontal components and all the critical loci of $\xi_i$'s. Since the 1-handle attach does not change the value of $Z_n$, it suffices to check the lemma for a homotopy of $\Sigma$. 

Suppose that a horizontal component $H_0$ is a part of an AL-graph of type $\vec{k}$ and has at least one trivalent vertex. If $S=\{H_1,H_2,\ldots,H_r\}$ is a set of horizontal components for $(\xi_1,\ldots,\xi_{3n})$ having trivalent vertices such that the total number of trivalent vertices in the graphs of $S$ is exactly $2n$ and if $\Gamma$ is a labeled trivalent graph, then the generating function $F_\Gamma(S)$ of counts of AL-graphs from $\Gamma$ whose set of horizontal components agrees with $S$ is a rational function in $\Q(t_1)\otimes\cdots\otimes\Q(t_{3n})$. The reason for this is that there are finitely many ways of joining legs of $H_i$'s to form the uncolored 3-valent graph $\Gamma$. In each joining of a pair of legs, the weighted number of ways of joining the pair by AL-paths is a rational function, as shown in \cite[Lemma~4.2]{Wa2} or in Lemma~\ref{lem:ST(V)} above. Hence $F_\Gamma(S)$ is a finite sum of rational functions. Moreover, there are finitely many possibilities for the set $S$ that contain $H_0$. Therefore, $F_\Gamma(H_0)=\sum_{\{S:H_0\in S\}}F_\Gamma(S)$ is a rational function in $\Q(t_1)\otimes\cdots\otimes\Q(t_{3n})$. 

Now suppose that a homotopy of $\Sigma$ crosses a trivalent vertex $v\in H_0$. Then $F_\Gamma(\Sigma;\xi_1,\ldots,\xi_{3n})$ may change under the homotopy and the terms that may change at the crossing are the terms in $F_\Gamma(H_0)$. More precisely, if the three edges incident to $v$ are labeled $i,j,k$, then all the terms in $F_\Gamma(H_0)$ get multiplied by $t_i^{\ve_i}t_j^{\ve_j}t_k^{\ve_k}$, $\ve_i,\ve_j,\ve_k\in\{-1,1\}$ (depending on the edge orientations). See Figure~\ref{fig:holonomy}. Namely, the change of $F_\Gamma(\Sigma;\xi_1,\ldots,\xi_{3n})$ under the homotopy at the crossing is $(t_i^{\ve_i}t_j^{\ve_j}t_k^{\ve_k}-1)\,F_\Gamma(H_0)$, whose trace vanishes by the Holonomy relation for $\widehat{\Lambda}$-colored graph (Figure~\ref{fig:relations}). This completes the proof.
\begin{figure}
\fig{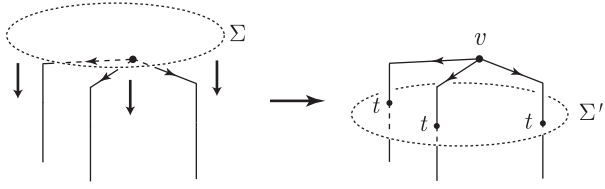}
\caption{Change under a homotopy of $\Sigma$}\label{fig:holonomy}
\end{figure}
\end{proof}

%%%%%%%%%%%%%%%%%%%%%%%%%%%%%%
\subsection{Bifurcation of oriented fiberwise Morse functions and their fiberwise gradients}\label{ss:bifurcation}

Let $\kappa:M\to S^1$ be a fiber bundle with fiber diffeomorphic to a connected oriented closed surface. A concordance of an oriented fiberwise Morse function gives a 2-parameter family $\wf:M\times[0,1]\to \R$ of oriented GMF's on a surface parametrized over $S^1\times[0,1]$. 
\begin{Lem}\label{lem:arrange-bd}
After a perturbation of the concordance $\wf$ fixing the endpoints, we may assume that there is a sequence $0<s_1<s_2<\ldots<s_r<1$ such that
\begin{enumerate}
\item $f_s$ is an oriented fiberwise Morse function if $s\neq s_1,s_2,\ldots,s_r$, and
\item at $s=s_i$, there is exactly one locus of $A_2$-singularities (birth-death locus) for $f_{s_i}$ which forms a finite covering over $S^1\times \{s_i\}$ by $\kappa\times\mathrm{id}$. 
\end{enumerate}
\end{Lem}
\begin{proof}
The proof is an analogue of the Beak lemma in \cite[Ch.~IV, \S3]{Ce}. Since $M\times [0,1]$ is 4-dimensional, a birth-death locus is 1-dimensional, and a critical locus of $\wf$ is 2-dimensional and unknotted, one may take a homotopy of a birth-death locus into a slice $M\times\{s\}$ that is disjoint from critical loci of $\wf$ and from other birth-death loci. For a birth locus, we may take a homotopy which pulls back the locus to a minimal parameter of $s$ on the locus. For a death locus, we may take a homotopy which pushes forward the locus to a maximal parameter of $s$ on the locus. Then the result is as desired. 
\end{proof}
To prove the invariance of $\wZ_n$, we take a concordance $\wf$ as in Lemma~\ref{lem:arrange-bd} and decompose $[0,1]$ as
\[ [0,s_1-\ve]\cup[s_1-\ve,s_1+\ve]\cup [s_1+\ve,s_2-\ve]\cup \cdots\cup
[s_r-\ve,s_r+\ve]\cup [s_r+\ve,1], \]
for a small $\ve>0$, and show that $\wZ_n$ is invariant at each piece. 

The restriction of $\wf$ on $M\times [s_{i-1}+\ve,s_i-\ve]$ gives a concordance through oriented fiberwise Morse functions. Let $\widetilde{\xi}=\{\xi_s\}_{s\in [0,1]}$ be the fiberwise gradient of $\wf$. We say that a parameter $s\in [s_{i-1}+\ve,s_i-\ve]$ is a {\it bifurcation} for $(\widetilde{f},\widetilde{\xi})=\{(f_s,\xi_s)\}_{s\in[0,1]}$ if $(f_s,\xi_s)$ does {\it not} satisfy one of the following conditions, which were assumed in the definition of the moduli space $\bacalM_{K_2}(\xi)$.
\begin{enumerate}
\item Level-exchange points and $1/1$-intersections occur at different levels for $\kappa$. 
\item Transversality of curves in the graphic.
\item $(f_s,\xi_s)$ satisfies parametrized Morse--Smale condition.
\end{enumerate}
In \cite{Ce, HW}, generic 2-parameter families of smooth functions are studied. According to \cite{Ce, HW}, a generic homotopy between generic loops in the space of Morse functions is as follows.
\begin{Lem}\label{lem:generic-ho}
Let $(\wf,\widetilde{\xi})$ be as above. For a generic choice of a fiberwise gradient $\widetilde{\xi}$ of $\widetilde{f}$ on $M\times[s_{i-1}+\ve,s_i-\ve]$, we may arrange that the possible bifurcations in the family $(\widetilde{f},\widetilde{\xi})$ on $[s_{i-1}+\ve,s_i-\ve]$ are of the following forms:
\begin{enumerate}
\item[(a)] A nondegenerate critical point on a level exchange curve.
\item[(b)] A point where three critical loci have the same value.
\item[(c)] A crossing between two level exchange curves.
\item[(d)] A nondegenerate critical point on a $1/1$-intersection curve.
\item[(e)] A crossing between two $1/1$-intersection curves (not successive).
\item[(f)] A crossing between two $1/1$-intersection curves (successive).
\item[(g)] A crossing between a level exchange curve and a $1/1$-intersection curve.
\end{enumerate}
 (See Figure~\ref{fig:bifurcations} and \ref{fig:bifurcations2}). Moreover, we may assume that no two bifurcations occur simultaneously in each time in $[s_{i-1}+\ve,s_i-\ve]$. 
\end{Lem}
\begin{figure}
\fig{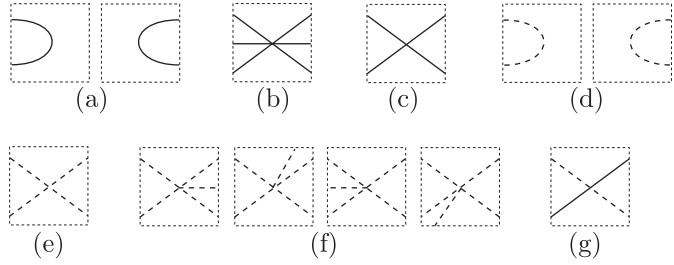}
\caption{Solid lines are level exchange parameters and dashed lines are parameters of $1/1$-intersections}\label{fig:bifurcations}
\end{figure}
\begin{figure}
\fig{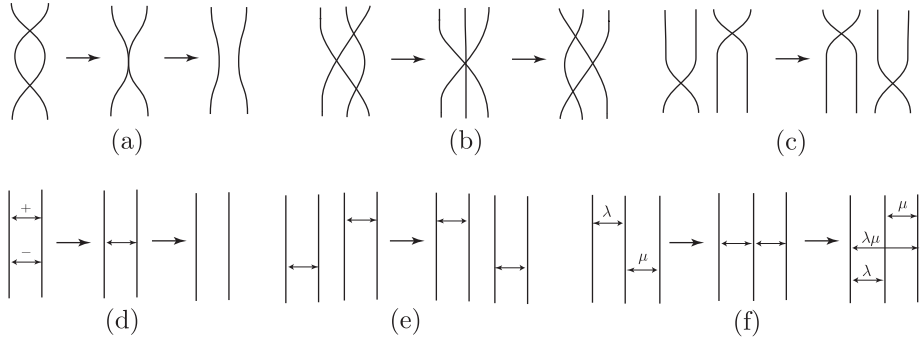}
\caption{Examples of changes in the graphics at the bifurcations. $\lambda,\mu=\pm 1$.}\label{fig:bifurcations2}
\end{figure}
In the rest of this section, let $(\widetilde{f},\widetilde{\xi})=\{(f_s,\xi_s)\}_{s\in [0,1]}$ be a generic concordance arranged as in Lemma~\ref{lem:arrange-bd}, \ref{lem:generic-ho}.

%%%%%%%%%%%%%%%%%%%%%%%%%%%%%%%%%%%%%%%%%%%%%%%%
\subsection{Invariance under a homotopy without bifurcations}

The moduli space $\bacalM_{K_2}(\xi)$ is defined by taking an auxiliary decomposition of $M$ into small pieces $W_i^{(j)}$, which are called {\it cells} (\cite[\S{2}]{Wa2}). The decomposition arose from the graphic. See Figure~\ref{fig:cellular-diag} for an example. The decomposition is done roughly as follows. First, partition $M$ by (horizontal) level surfaces $\kappa^{-1}(c)$ for several values of $c$ so that each piece $B$ satisfies either of the following.
\begin{enumerate}
\item $B$ contains exactly one level exchange pair of critical loci and no $1/1$-intersections.
\item $B$ contains no level exchange pair of critical loci but may contain several $1/1$-intersections.
\end{enumerate}
Then partition each piece $B$ further by (vertical) level surfaces of $f$ so that each piece contains one or two components of critical loci and that it contains two if and only if the two loci are the pair forming a crossing in the graphic. 

By using the decomposition, we may identify each AL-path $\gamma$ in $\bacalM_{K_2}(\xi)$ with a point of the direct product of vertical level surfaces on which $\gamma$ intersects. Thus one sees that $\bacalM_{K_2}(\xi)$ is locally modeled on the direct product of the vertical level surfaces. Namely, embed a small neighborhood of $\gamma\in \acalM_{K_2}(\xi)$ into the product of the vertical level surfaces and take the closure in the product. Then glue together such local models for the closure suitably to obtain $\bacalM_{K_2}(\xi)$.

We orient $M\times J$ by $o(M\times J)=o(J)\wedge o(M)$.
\begin{figure}
\fig{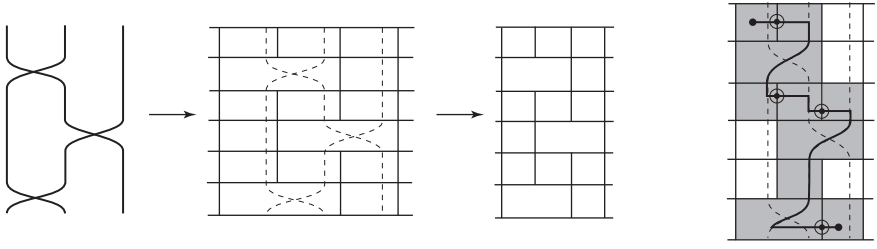}
\caption{A decomposition of $M$ into cells corresponding to the graphic. An AL-path $\gamma$ is characterized by a sequence of points at which $\gamma$ and the vertical walls intersect.}\label{fig:cellular-diag}
\end{figure}

\begin{Lem}\label{lem:nobifurcations}
Let $J=[\alpha,\beta]\subset [0,1]$ be an interval such that the restriction of $(\widetilde{f},\widetilde{\xi})$ on $M\times J$ does not have bifurcations. Then 
\[ \wZ_n(\xi_{\alpha},\xi_2,\ldots,\xi_{3n},\mathfrak{s})=\wZ_n(\xi_{\beta},\xi_2,\ldots,\xi_{3n},\mathfrak{s}). \]
\end{Lem}
\begin{proof}
Since there are no bifurcations in the restriction $(\widetilde{f}_J,\widetilde{\xi}_J)$ of $(\widetilde{f},\widetilde{\xi})$ on $M\times J$, the graphics of $f_\alpha$ and $f_\beta$ with the arrows of 1/1-intersections are isotopic through those in the 1-parameter family. Hence, the auxiliary decompositions of $M$ for $\xi_\alpha$ and $\xi_\beta$ are extended to a decomposition of $M\times J$, which is isomorphic to the product decomposition with $J$. By using the product structure, we may define a moduli space $\bacalM_{K_2}(\widetilde{\xi}_J)$, which gives a cobordism between $\bacalM_{K_2}(\xi_\alpha)$ and $\bacalM_{K_2}(\xi_\beta)$ diffeomorphic to $\bacalM_{K_2}(\xi_\alpha)\times J$. By replacing $\bacalM_{K_2}(\xi_\alpha)$ with $\bacalM_{K_2}(\widetilde{\xi}_J)$ in the definition of $Z_n$, we may define the 1-chain
\[ Z_n(\widetilde{\xi}_J,\xi_2,\ldots,\xi_{3n})=
\sum_\Gamma \Tr_\Gamma\langle Q^\circ(\widetilde{\xi}_J),Q^\circ(\xi_2)\times J,\ldots,Q^\circ(\xi_{3n})\times J\rangle_\Gamma \]
in $C_1(\bConf_{2n}(M)\times J)\otimes_\Q \calA_n(\widehat{\Lambda})$, where $Q^\circ(\widetilde\xi_J)$ is the corresponding bordism between $Q^\circ(\xi_\alpha)$ and $Q^\circ(\xi_\beta)$. The coorientation of $Q^\circ(\widetilde{\xi}_J)$ (resp. $Q^\circ(\xi)_j\times J$) is given so that its restrictions on $\bConf_{K_2}(M)\times \{\alpha\}$ is equivalent to that of $Q^\circ(\xi_\alpha)$ (resp. $Q^\circ(\xi_j)$). 

For $i=2,\ldots,3n$, let $\xi_i(J):M\times J\to TM\times J$ be the vector fields defined by $\xi_i(J)(x,s)=(\xi_i(x),x,s)$. 
Putting $\vec{\xi}(J)=(\widetilde{\xi}_J,\xi_2(J),\ldots,\xi_{3n}(J))$, we define
\[ Z_n^\anomaly(\vec{\xi}(J))
=\sum_\Gamma \#\calM_{\Gamma}^\loc(-\vec{\xi}(J))\,[\Gamma(1,\ldots,1)], \]
where $\#\calM_{\Gamma}^\loc(-\vec{\xi}(J))$ is the moduli space of linear graphs in the rank 3 vector bundle $TM\times J$ over $M\times J$. We shall check that the homology class of $\partial Z_n(\widetilde{\xi}_J,\xi_2,\ldots,\xi_{3n})$ is given by 
\[ Z_n(\xi_\alpha,\xi_2,\ldots,\xi_{3n})-Z_n(\xi_\beta,\xi_2,\ldots,\xi_{3n})+Z_n^\anomaly(\vec{\xi}(J)).\]
This is an analogue of an identity in the proof of \cite[Lemma~10.1]{Wa1}. By assuming this, the proof will complete as follows.
\[ \begin{split}
  &\wZ_n(\xi_\alpha,\xi_2,\ldots,\xi_{3n},\mathfrak{s})-\wZ_n(\xi_\beta,\xi_2,\ldots,\xi_{3n},\mathfrak{s})\\
  &=Z_n(\xi_\alpha,\xi_2,\ldots,\xi_{3n})-Z_n(\xi_\beta,\xi_2,\ldots,\xi_{3n})-Z_n^\anomaly(\vec{\rho}_W(\alpha))+Z_n^\anomaly(\vec{\rho}_W(\beta))\\
  &=Z_n(\xi_\alpha,\xi_2,\ldots,\xi_{3n})-Z_n(\xi_\beta,\xi_2,\ldots,\xi_{3n})+Z_n^\anomaly(\vec{\xi}(J))\\
  &=[\partial Z_n(\widetilde{\xi}_J,\xi_2,\ldots,\xi_{3n})]=0,
\end{split}\]
where $\vec{\rho}_W(\alpha)$ and $\vec{\rho}_W(\beta)$ are the tuples of sections of $T^vW$ that extend $(\xi_\alpha,\xi_2,\ldots,\xi_{3n})$ and $(\xi_\beta,\xi_2,\ldots,\xi_{3n})$ respectively. 

Since $\langle Q^\circ(\widetilde{\xi}_J),Q^\circ(\xi_2)\times J,\ldots,Q^\circ(\xi_{3n})\times J\rangle_\Gamma$ is the intersection form among relative cycles in $(\bConf_\Gamma(M)\times J,\partial(\bConf_\Gamma(M)\times J))$, the boundary of $\langle Q^\circ(\widetilde{\xi}_J),Q^\circ(\xi_2)\times J,\ldots,Q^\circ(\xi_{3n})\times J\rangle_\Gamma$ that are not on $\bConf_{\Gamma}(M)\times\partial J$ consists of configurations for AL-graphs that are in the preimage of $\partial \bConf_{2n}(M)$. According to the description of the strata of $\bConf_n(M)$ in \S\ref{ss:FM}, each of such AL-graphs $I:\Gamma\to M$ is one of the following forms.
\begin{enumerate}
\item (Principal face) For an edge $e\in E(\Gamma)$, the image $I(e)$ of $e$ collapses into a point. 
\item (Hidden face) A subgraph $T_1$ in a horizontal component $T$ of $I(\Gamma)$ collapses into a point. 
\end{enumerate}
Note that it is not necessary to consider an AL-graph with a non self-loop edge forming a closed AL-path in $M$ since such an AL-graph and an AL-graph with one 4-valent vertex do not occur simultaneously in a generic homotopy. If $\Gamma$ is the dumbbell graph $\bigcirc$\kern-.3mm---\kern-.3mm$\bigcirc$, there may be an AL-graph formed by exactly two closed AL-paths sharing a point, which is treated in case (1). In a generic homotopy, there does not exist an AL-graph formed by three closed AL-paths sharing one point. 

The contributions of the case (1) are canceled each other out in $Z_n(\widetilde{\xi}_J,\xi_2,\ldots,\xi_{3n})$ by the IHX relation, for a similar reason as in Chern--Simons perturbation theory for homology 3-spheres (e.g., \cite{Ko, KT, Les1, Wa1}). More precisely, for each labeled trivalent graph $\Gamma$, the contributions in $F_\Gamma(\Sigma;\widetilde{\xi}_J,\xi_2,\ldots,\xi_{3n})$ having such collapsed edge, say, the one labeled $k$, is a rational function with no terms of nonzero exponents of $t_k$, as in the proof of Lemma~\ref{lem:indep_sigma}. Thus the IHX relation for $\widehat{\Lambda}$-colored graphs can be applied to prove that the changes are canceled each other out without any difficulty. See \cite[Lemma~10.1]{Wa1} for detail. 

The contributions of the case (2) for $T_1\neq T$ vanish by an analogue of Kontsevich's lemma \cite[Lemma~2.1]{Ko} (see also \cite[Lemma~6.3, 6.4, 6.5]{Wa1}). Roughly, the contribution of the case (2) is the product of the count of AL-graphs from $\Gamma/T_1$ in $\bConf_{\Gamma/T_1}(M)$ and the count of linear graphs from $T_1$ in $\bConf_{|V(T_1)|}^\loc(\R^3)$. $T_1$ has a univalent vertex or a bivalent vertex. If $T_1$ has a univalent vertex, then the moduli space of linear graphs in $\R^3$ is generically empty by a dimensional reason. If $T_1$ has a bivalent vertex, then there is an orientation reversing involution on $\bConf_{|V(T_1)|}^\loc(\R^3)$, which exchanges the moduli space of $T_1$ and that of another graph $T_1^*$, which is $T_1$ with different labels and different edge-orientations. The cancellation between the two graphs $T_1$ and $T_1^*$ is examined in \cite[Lemma~6.3]{Wa1}. 

The contributions of the case (2) for $T_1=\Gamma$, which are not covered by the previous paragraph, are those correspond to the collapse of whole graph. They contribute as $Z_n^\anomaly(\vec{\xi}(J))$. 
\end{proof}
%\clearpage

%%%%%%%%%%%%%%%%%%%%%%%%%%%%%%%%%%%%%%%%%%%
\subsection{Invariance at bifurcations of level-exchange loci}

\begin{Lem}\label{lem:case(a)}
Let $s_0\in [0,1]$ be a parameter at which the bifurcation of type (a) in Figure~\ref{fig:bifurcations} occurs in $(\wf,\widetilde{\xi})$. For a small number $\ve>0$, 
\[ \wZ_n(\xi_{s_0-\ve},\xi_2,\ldots,\xi_{3n},\mathfrak{s})=\wZ_n(\xi_{s_0+\ve},\xi_2,\ldots,\xi_{3n},\mathfrak{s}). \]
\end{Lem}
\begin{proof}
Let $(u_0,s_0)\in S^1\times [0,1]$ be a point such that the projections of two critical loci of $f_{s_0}|_{\Sigma_{u_0}}$ in the graphic are tangent to each other (see Figure~\ref{fig:bifurcations2} (a)). After a small perturbation of $\widetilde{\xi}$ which does not change the critical loci, we may assume that the moduli space $\calM_\Gamma^\loc(-\xi_{s_0},-\xi_2,\ldots,-\xi_{3n})$ is empty on $[s_0-\ve,s_0+\ve]$ for all $\Gamma$. Hence we may assume that the anomaly correction term $Z_n^\anomaly$ does not change through $[s_0-\ve,s_0+\ve]$. 

Take a neighborhood $U_{(u_0,s_0)}$ of $(u_0,s_0)$ in $S^1\times [0,1]$ so small that there are no $1/1$-intersections in $(\kappa\times\mathrm{id})^{-1}U_{(u_0,s_0)}\subset M\times [0,1]$. Let $J=[s_0-\ve,s_0+\ve]$ and suppose that $\ve$ is sufficiently small so that $[u_0-\ve,u_0+\ve]\times J\subset U_{(u_0,s_0)}$. Moreover, we assume without loss of generality that the curve of the level exchange parameters in $U_{(u_0,s_0)}$ intersects the segment $\{s_0-\ve\}\times [u_0-\ve,u_0+\ve]$ in two points, as in the left hand side of Figure~\ref{fig:bifurcations} (a). 

We may decompose $M\times J\setminus \mathrm{Int}\,(\kappa\times\mathrm{id})^{-1}([u_0-\ve,u_0+\ve]\times J)$ into small pieces as in the proof of Lemma~\ref{lem:nobifurcations} since $\widetilde{\xi}$ has no bifurcation there. The rest $(\kappa\times\mathrm{id})^{-1}([u_0-\ve,u_0+\ve]\times J)$ can be decomposed as follows. For the critical loci of $f_J|_{\kappa^{-1}[u_0-\ve,u_0+\ve]}$ that are not involved in the bifurcation, we may take product cells as in the proof of Lemma~\ref{lem:nobifurcations} each of which includes exactly one critical locus. We consider the closure $W_{(c_0,s_0)}$ of the complement in $M\times J$ of the union of all the cells taken above as one big cell, which includes exactly two critical loci. See Figure~\ref{fig:partition}.
\begin{figure}
\fig{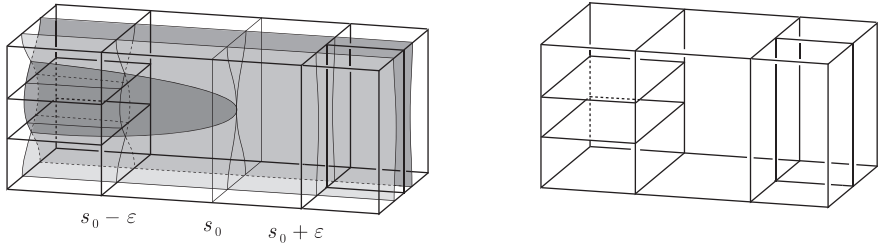}
\caption{The cell decomposition at bifurcation of type (a).}\label{fig:partition}
\end{figure}

Now we have a decomposition of $M\times J$ into cells. Applying a similar method as in the definition of $\bacalM_{K_2}(\xi_{s_0-\ve})$ with the decomposition of $M\times J$ given above, we obtain a moduli space $\bacalM_{K_2}(\widetilde{\xi}_J)$ of AL-paths in $M\times J$, which restricts to $\bacalM_{K_2}(\xi_{s_0-\ve})$ and $\bacalM_{K_2}(\xi_{s_0+\ve})$ at $s=s_0\pm \ve$, without any difficulty. By replacing $\bacalM_{K_2}(\xi_\alpha)$ with $\bacalM_{K_2}(\widetilde{\xi}_J)$ in the definition of $Z_n$, we may define the 1-chain
\[ 
  Z_n(\widetilde{\xi}_J,\xi_2,\ldots,\xi_{3n})=
\sum_\Gamma \Tr_\Gamma\langle Q^\circ(\widetilde{\xi}_J),Q^\circ(\xi_2)\times J,\ldots,Q^\circ(\xi_{3n})\times J\rangle_\Gamma 
\]
in $C_1(\bConf_{2n}(M)\times J)\otimes_\Q \calA_n(\widehat{\Lambda})$. Proof of the identity 
\[ \begin{split}
0=[\partial Z_n(\widetilde{\xi}_J,\xi_2,\ldots,\xi_{3n})]=&Z_n(\xi_{s_0-\ve},\xi_2,\ldots,\xi_{3n})-Z_n(\xi_{s_0+\ve},\xi_2,\ldots,\xi_{3n})
\end{split}\]
is the same as Lemma~\ref{lem:nobifurcations}.
\end{proof}

\begin{Lem}\label{lem:case(b)}
Let $s_0\in [0,1]$ be a parameter at which the bifurcation of type (b) or (c) in Figure~\ref{fig:bifurcations} occurs in $(\wf,\widetilde{\xi})$. For a small number $\ve>0$, 
\[\wZ_n(\xi_{s_0-\ve},\xi_2,\ldots,\xi_{3n},\mathfrak{s})=\wZ_n(\xi_{s_0+\ve},\xi_2,\ldots,\xi_{3n},\mathfrak{s}).\]
\end{Lem}
\begin{proof}
The proof is the same as Lemma~\ref{lem:case(a)}. Note that for (c), there are two big cells at $s=s_0$ in the decomposition of $M$ into cells.
\end{proof}

%\clearpage

%%%%%%%%%%%%%%%%%%%%%%%%%%%%%%%%%%%%
\subsection{Invariance at bifurcations of $1/1$-intersection loci}

\begin{Lem}\label{lem:case(d)}
Let $s_0\in [0,1]$ be a parameter at which the bifurcation of type (d) in Figure~\ref{fig:bifurcations} occurs in $(\wf,\widetilde{\xi})$. For a small number $\ve>0$, 
\[ \wZ_n(\xi_{s_0-\ve},\xi_2,\ldots,\xi_{3n},\mathfrak{s})=\wZ_n(\xi_{s_0+\ve},\xi_2,\ldots,\xi_{3n},\mathfrak{s}).\]
\end{Lem}
\begin{proof}
Let $(u_0,s_0)\in S^1\times [0,1]$ be a point such that the bifurcation of type (d) occurs. After a small perturbation of $\widetilde{\xi}$ which does not change the bifurcation parameters, we may assume that the moduli space $\calM_\Gamma^\loc(-\xi_{s_0},-\xi_2,\ldots,-\xi_{3n})$ is empty on $[s_9-\ve,s_0+\ve]$ for all $\Gamma$. Hence we may assume that the anomaly correction term $Z_n^\anomaly$ does not change through $[s_0-\ve,s_0+\ve]$. 

Take a neighborhood $U_{(u_0,s_0)}$ of $(u_0,s_0)$ in $S^1\times [0,1]$ so small that there are no other $1/1$-intersections nor level exchange bifurcations in $(\kappa\times\mathrm{id})^{-1}U_{(u_0,s_0)}\subset M\times [0,1]$. Let $J=[s_0-\ve,s_0+\ve]$ and suppose that $\ve$ is sufficiently small so that $[u_0-\ve,u_0+\ve]\times J\subset U_{(u_0,s_0)}$. We assume without loss of generality that the curve of the $1/1$-intersection parameters in $U_{(u_0,s_0)}$ intersects the segment $\{s_0-\ve\}\times [u_0-\ve,u_0+\ve]$ in two points, as in the left hand side of Figure~\ref{fig:bifurcations} (d). Then we may take a decomposition of $M\times J$ into pieces as in the proof of Lemma~\ref{lem:nobifurcations} with no trouble.

By the parametrized Morse--Smale condition for $\widetilde{\xi}$, we may apply a similar method as in the definition of $\bacalM_{K_2}(\xi_{s_0-\ve})$ with the decomposition of $M\times J$ given above. Then we obtain a moduli space $\bacalM_{K_2}(\widetilde{\xi}_J)$ of AL-paths in $M\times J$, which restricts to $\bacalM_{K_2}(\xi_{s_0-\ve})$ and $\bacalM_{K_2}(\xi_{s_0+\ve})$ at $s=s_0\pm \ve$. By replacing $\bacalM_{K_2}(\xi_\alpha)$ with $\bacalM_{K_2}(\widetilde{\xi}_J)$ in the definition of $Z_n$, we may define the 1-chain
\[ 
Z_n(\widetilde{\xi}_J,\xi_2,\ldots,\xi_{3n})=
\sum_\Gamma \Tr_\Gamma\langle Q^\circ(\widetilde{\xi}_J),Q^\circ(\xi_2)\times J,\ldots,Q^\circ(\xi_{3n})\times J\rangle_\Gamma
\]
in $C_1(\bConf_{2n}(M)\times J)\otimes_\Q \calA_n(\widehat{\Lambda})$. Proof of the identity
\[ \begin{split}
0=[\partial Z_n(\widetilde{\xi}_J,\xi_2,\ldots,\xi_{3n})]=&Z_n(\xi_{s_0-\ve},\xi_2,\ldots,\xi_{3n})-Z_n(\xi_{s_0+\ve},\xi_2,\ldots,\xi_{3n})
\end{split}\]
is the same as Lemma~\ref{lem:nobifurcations}.
\end{proof}

\begin{Lem}\label{lem:case(e)}
Let $s_0\in [0,1]$ be a parameter at which the bifurcation of type (e) or (g) in Figure~\ref{fig:bifurcations} occurs in $(\wf,\widetilde{\xi})$. For a small number $\ve>0$, 
\[ \wZ_n(\xi_{s_0-\ve},\xi_2,\ldots,\xi_{3n},\mathfrak{s})=\wZ_n(\xi_{s_0+\ve},\xi_2,\ldots,\xi_{3n},\mathfrak{s}).\]
\end{Lem}
\begin{proof}
The proof is the same as Lemma~\ref{lem:case(d)}.
\end{proof}

\begin{Lem}\label{lem:case(f)}
Let $s_0\in [0,1]$ be a parameter at which the bifurcation of type (f) in Figure~\ref{fig:bifurcations} occurs in $(\wf,\widetilde{\xi})$. For a small number $\ve>0$, 
\[ \wZ_n(\xi_{s_0-\ve},\xi_2,\ldots,\xi_{3n},\mathfrak{s})=\wZ_n(\xi_{s_0+\ve},\xi_2,\ldots,\xi_{3n},\mathfrak{s}). \]
\end{Lem}
\begin{proof}
After a small perturbation of $\widetilde{\xi}$ which does not change the bifurcation parameters, we may assume that the moduli space $\calM_\Gamma^\loc(-\xi_{s_0},-\xi_2,\ldots,-\xi_{3n})$ is empty on $[s_0-\ve,s_0+\ve]$ for all $\Gamma$. Hence we may assume that the anomaly correction term $Z_n^\anomaly$ does not change through $[s_0-\ve,s_0+\ve]$. 

Let $J=[s_0-\ve,s_0+\ve]$. As in the proof of Lemma~\ref{lem:case(d)}, we may define a moduli space $\bacalM_{K_2}(\widetilde{\xi}_J)$ of AL-paths in $M\times \{s\}$, $s\in J$, which restricts to $\bacalM_{K_2}(\xi_{s_0-\ve})$ and $\bacalM_{K_2}(\xi_{s_0+\ve})$ at $s=s_0\pm \ve$. This gives rise to the 1-chain
\[ Z_n(\widetilde{\xi}_J,\xi_2,\ldots,\xi_{3n})=
\sum_\Gamma \Tr_\Gamma\langle Q^\circ(\widetilde{\xi}_J),Q^\circ(\xi_2)\times J,\ldots,Q^\circ(\xi_{3n})\times J\rangle_\Gamma\]
in $C_1(\bConf_{2n}(M)\times J)\otimes_\Q \calA_n(\widehat{\Lambda})$. Note that the the boundary of the chain\\ $\langle Q^\circ(\widetilde{\xi}_J),Q^\circ(\xi_2)\times J,\ldots,Q^\circ(\xi_{3n})\times J\rangle_\Gamma$ consist only of the forms given in the proof of Lemma~\ref{lem:nobifurcations} although the trajectory spaces for the $1/1$-intersections involved in the bifurcation may have boundaries at $s=s_0$. The AL-graphs in $s<s_0$ that may arrive at the boundary at $s=s_0$ are those with an edge $e$ including both the horizontal segments labeled $\lambda$ and $\mu$ in the left side of Figure~\ref{fig:bifurcations2} (f). They can be paired with the AL-graphs in $s>s_0$ with an edge $e'$ including the horizontal segment labeled $\lambda\mu$ in the right side of Figure~\ref{fig:bifurcations2} (f). 
\[ \fig{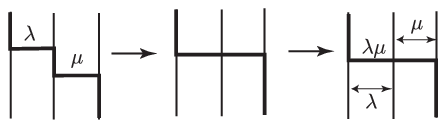} \]
Since these edges have the same sign, the boundaries at $s=s_0$ are cancelled with each other. Hence we have 
\[ \begin{split}
0=[\partial Z_n(\widetilde{\xi}_J,\xi_2,\ldots,\xi_{3n})]=&Z_n(\xi_{s_0-\ve},\xi_2,\ldots,\xi_{3n})-Z_n(\xi_{s_0+\ve},\xi_2,\ldots,\xi_{3n}).
\end{split}\]
\end{proof}

\subsection{Invariance at birth-death locus}

\begin{Lem}\label{lem:bd}
Let $s_0\in [0,1]$ be a parameter at which a birth-death bifurcation occurs in $(\wf,\widetilde{\xi})$. For a small number $\ve>0$, 
\[ \wZ_n(\xi_{s_0-\ve},\xi_2,\ldots,\xi_{3n},\mathfrak{s})=\wZ_n(\xi_{s_0+\ve},\xi_2,\ldots,\xi_{3n},\mathfrak{s}).\]
\end{Lem}
\begin{proof}
By the symmetry between a birth and a death bifurcation with respect to the parameter $s$, it is enough to check the lemma only for a birth bifurcation. Let $\gamma\subset M\times [0,1]$ be the birth locus that occurs at $s_0$ and suppose that a pair $(\alpha,\beta)$ of critical loci of $\wf$ of index 1 and 2 appear after $s_0$. See Figure~\ref{fig:bd-cancel} (a). The case where $\mathrm{ind}\,\alpha=1$ and $\mathrm{ind}\,\beta=0$ is symmetric to this case. Moreover, we assume for simplicity that the image of $\gamma$ in $S^1\times[0,1]$ spans $H_1(S^1\times[0,1];\Z)=\Z$ since the proof of other cases are the same. Take a small neighborhood $U_0$ of $S^1\times \{s_0\}$ in $S^1\times[0,1]$, a section $\widetilde{\gamma}:U_0\to (\kappa\times\mathrm{id})^{-1}(U_0)$ that is a smooth extension of $\gamma$, and a fiberwise small tubular neighborhood $U$ of $\widetilde{\gamma}$ in $M\times [0,1]$. Choose a level surface locus $T$ of $\wf$ in $U$ that is disjoint from $\alpha\cup\beta\cup\gamma$ in $U$ and that lies just below $(\alpha\cup\beta\cup\gamma)\cap U$. 
\begin{figure}
\fig{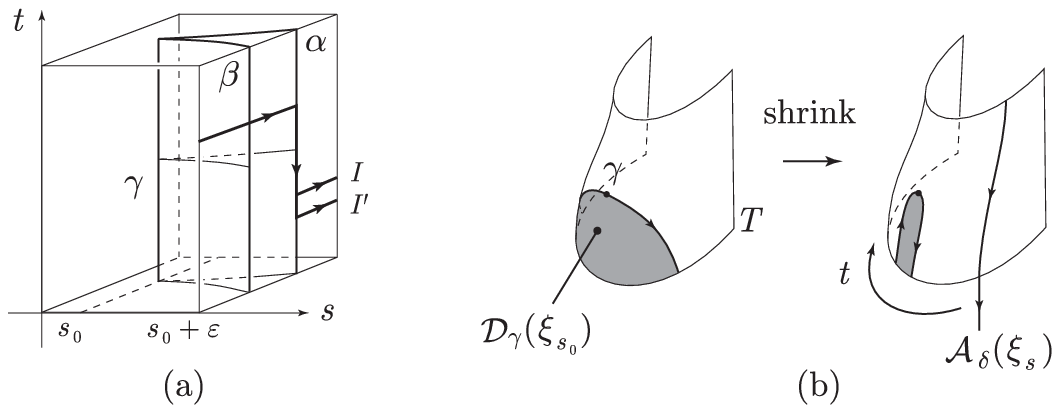}
\caption{}\label{fig:bd-cancel}
\end{figure}

By the Normal form lemma for $A_2$-singularities in \cite{Ig1}, we may assume that for each point $(s_0,u)\in U$ on $\gamma$, there is a local coordinate $(s,t,x_1,x_2)$ around $(s_0,u)$ such that $\wf$ agrees with
\[ \wf(s,t,x_1,x_2)=g(s,t)+x_1^3-(s-s_0)x_1-x_2^2 \]
for a smooth function $g(s,t)$, where it makes sense. Moreover, we may assume for a fixed pair $(s_0,u)$ that the Euclidean local coordinate is isometric near $(s_0,u)$. 

The intersection of the descending manifold locus $\wcalD_\gamma(\widetilde{\xi})$ of $\gamma$ with $T$ forms a bundle over $S^1\times\{s_0\}$ with fiber a small closed interval. By perturbing $\widetilde{\xi}$ in a small neighborhood of $T$, the closed interval can be made arbitrarily small in each fiber. Then by smoothness of $\widetilde{\xi}$, the intersection of $\wcalD_\alpha(\widetilde{\xi})\cup \wcalD_\beta(\widetilde{\xi})$ in each fiber of $U$ with $T$ can be made arbitrarily small too. Then for a small number $\ve>0$, $\wcalD_\alpha(\xi_{s_0+\ve})\cup \wcalD_\beta(\xi_{s_0+\ve})$ forms a 1-parameter family of thin half-disks which flow down to a critical locus of index 0 for a generic parameter, or which may go through another critical locus of index 1 as $t$ increases. See Figure~\ref{fig:bd-cancel} (b).

%By using this local model, the moduli space of AL-paths $\bacalM_{K_2}(\xi_{s_0-\ve})$ extends over $J=[s_0-\ve,s_0+\ve]$. 

First, we consider the limit of AL-graphs as $s\to s_0$ from below. If $\widetilde{\xi}$ is generic, the moduli spaces $\bcalM_{\Gamma(\vec{k})}^\mathrm{AL}(\Sigma;\xi_s,\xi_2,\ldots,\xi_{3n})$ for $s\in [s_0-\ve,s_0)$ forms a finite covering over $[s_0-\ve,s_0)$. As seen from the result in \cite[\S{8.4}]{Wa1}, the limit consists of AL-graphs that do not intersect $\gamma$ and those with a broken edge that intersect $\gamma$. The AL-graphs at $s=s_0$ that do not intersect $\gamma$ extends smoothly over $(s_0,s_0+\ve]$ as a covering over $[s_0-\ve,s_0+\ve]$. The AL-graphs that intersect $\gamma$ stop at $s=s_0$ as boundaries of a 1-dimensional moduli space and do not extend over $(s_0,s_0+\ve]$. We shall check that $\widetilde{\xi}$ can be perturbed within the space of 1-parameter families of gradient-like vector fields for the given a 2-parameter family $\wf$ of GMF's so that all the boundaries of the moduli spaces at $s=s_0$ disappear. 

Long broken edge: suppose that the limit of a 1-parameter family $\{I_s\}_{s\in [s_0-\ve,s_0)}$ of AL-graphs at $s_0$ has an edge $e$ with positive length that is broken at the birth locus $\gamma$. Let $e'$ be the horizontal segment in $e$ that is broken at $\gamma$. In this case, by perturbing $\widetilde{\xi}$ in small neighborhoods of two points on $e'$ that are not close to $\gamma$, we may assume that $e$ is disjoint from $\gamma$. This is possible since we are assuming that the descending manifold locus $\wcalD_\gamma(\widetilde{\xi})$ forms a bundle of arbitrarily thin half-disks. 
\[ \fig{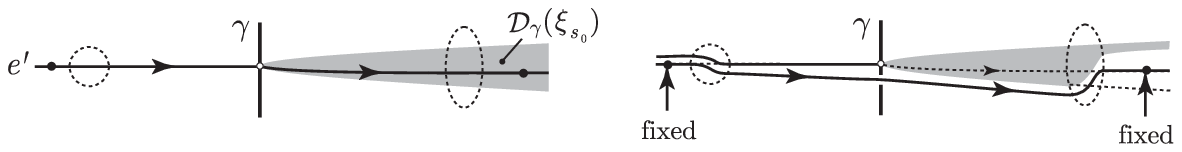} \]

Collapsed broken edge: suppose that the limit of a 1-parameter family $\{I_s\}_{s\in [s_0-\ve,s_0)}$ of AL-graphs at $s_0$ has an edge that collapses to a point on the birth locus $\gamma$. This is the contribution of $\partial\bConf_{K_2}(M)$. The limit can be described as follows. Let $J=[s_0-\ve,s_0+\ve]$. We may define the moduli space $\bacalM_{K_2}(\widetilde{\xi}_J)$ of AL-paths by gluing the closures in the auxiliary spaces ($\times J$) as in the definition of $\bacalM_{K_2}(\xi)$ in \cite[\S{2}]{Wa2}. The evaluation map $\bar{b}:\bacalM_{K_2}(\widetilde{\xi}_J)\to M^{K_2}\times J$ represents a 5-dimensional $\widehat{\Lambda}$-chain $Q(\widetilde{\xi}_J)$. The boundary of $Q(\widetilde{\xi}_J)$ on $\partial \bConf_{K_2}(M)\times J$ is described as follows. Let $(M\times J)_0=(M\times J)\setminus\bigcup_{\sigma\,:\,\mathrm{critical\, locus}}\sigma$ and let $s_{\widetilde{\xi}_J}:(M\times J)_0\to ST(M)\times J$ be the normalization $-\widetilde{\xi}_J/\|\widetilde{\xi}_J\|$ of the section $-\widetilde{\xi}_J$. The closure $\overline{s_{{\widetilde{\xi}_J}}((M\times J)_0)}$ in $ST(M)\times J$ is a smooth manifold with boundary whose boundary at $\partial\bConf_{K_2}(M)$ is the disjoint union of circle bundles over the critical loci $\sigma$ of $\widetilde{\xi}_J$ (including the birth-death locus $\gamma$), for a similar reason as \cite[Lemma~4.3]{Sh}. The fibers of the circle bundles are equators of the fibers of $ST(\sigma)$. Let $E^-_\sigma$ be the total space of the 2-disk bundle over $\sigma$ whose fibers are the lower hemispheres of the fibers of $ST(\sigma)$ which lie below the tangent spaces of the level surfaces of $\kappa$. Then $\partial\overline{s_{\widetilde{\xi}_J}((M\times J)_0)}=\bigcup_\sigma \partial E_\sigma^-$ as sets. Let
\[ s_{\widetilde{\xi}_J}^*(M\times J)=\overline{s_{\widetilde{\xi}_J}((M\times J)_0)}\cup \bigcup_\sigma E^-_\sigma\subset ST(M)\times J.\]
This is a 4-dimensional piecewise smooth manifold. We orient $s_{\widetilde{\xi}_J}^*(M\times J)$ by extending the natural orientation $(s_{\widetilde{\xi}_J}^{-1})^*o(M\times J)=(s_{\widetilde{\xi}_J}^{-1})^*o(J)\wedge o(M)$ on $s_{\widetilde{\xi}_J}((M\times J)_0)$. The contribution of $\partial\bConf_{K_2}(M)$ in the boundary of $Q(\widetilde{\xi}_J)$ is of the following form:
\[ -s_{\widetilde{\xi}_J}^*(M\times J)-\sum_{\sigma}(-1)^{\mathrm{ind}\,\sigma}\ve(\sigma)\,t^{p(\sigma)}\,ST(\sigma^\irr). \]
See \cite[\S{5.4}]{Wa1} for the reason of the signs in this formula. We could describe the sign $\ve(\sigma)$ and the orientations of $ST(\sigma^\irr)$ precisely, but it is not necessary here. The point is that one can define $Z_n^{\mathrm{anomaly}}(\xi_{s_0},\xi_2,\ldots,\xi_{3n},\mathfrak{s})$ using the restriction of $Q(\widetilde{\xi}_J)$ at $s=s_0$ and that there is no bifurcation for the anomaly correction term $Z_n^\anomaly$ around $s=s_0$ if $\widetilde{\xi}_J$ is generic, by a dimensional reason. This completes the proof of vanishing of the contributions of the AL-graphs with an edge degenerate at the birth locus.

Next, we consider the limit of AL-graphs as $s\to s_0$ from above. In this side, there may be other kind of AL-graphs with an edge broken at the birth point, namely, those with an edge that visits the critical locus $\alpha$ (of index 1). Let $I\in \acalM_{\Gamma(\vec{k})}(\Sigma;\xi_{s_0+\ve},\xi_2,\ldots,\xi_{3n})$ be a generic AL-graph whose first edge $e_1$ has a vertical segment $C$ included in the critical locus $\alpha_{s_0+\ve}=\alpha\cap (M\times\{s_0+\ve\})$. There must be a horizontal segment $L$ in $e_1$ next to $C$. Let $H_0$ be a horizontal component in $I$ that contains $L$. 

When $H_0$ has a trivalent vertex, let $H_0'$ be the graph obtained from $H_0$ by removing the edge $L$. Then the moduli space of AL-graphs $H_0'\to M\times J$ near the inclusion $H_0'\to H_0$ is 1-dimensional and the locus of the bivalent vertex of $H_0'$ forms a local section $\delta:U\times\{s_0+\ve\}\to M_U\times\{s_0+\ve\}$ of $\kappa_1\times\mathrm{id}:M\times\{s_0+\ve\}\to S^1\times\{s_0+\ve\}$, where $U$ is a small open interval in $S^1$, $\kappa_1$ is the first of $(\kappa_1,\kappa_2,\ldots,\kappa_{3n})$ (\S\ref{ss:def_Z}) and $M_U=\kappa_1^{-1}(U)$. We define the descending manifold $\wcalD_\delta^0(\xi_{s_0+\ve})$ and the ascending manifold $\wcalA_\delta^0(\xi_{s_0+\ve})$ of $\delta$ as follows.
\[ \begin{split}
  \wcalD_\delta^0(\xi_{s_0+\ve})&=\{(x,s_0+\ve)\in M_U\times\{s_0+\ve\};\exists\,T\geq 0,\Phi_{-\xi_{s_0+\ve}}^T(\delta(s))=x\},\\
  \wcalA_\delta^0(\xi_{s_0+\ve})&=\{(x,s_0+\ve)\in M_U\times\{s_0+\ve\};\exists\,T\leq 0,\Phi_{-\xi_{s_0+\ve}}^T(\delta(s))=x\}.
\end{split}\]
Then the segment $L$ corresponds to a transversal intersection of $\wcalD_\alpha(\xi_{s_0+\ve})$ with $\wcalA_\delta(\xi_{s_0+\ve})$. The observation in the previous paragraph implies that $L$ goes along one side of a thin half-disk in $\wcalD_\alpha(\xi_{s_0+\ve})\cup \wcalD_\beta(\xi_{s_0+\ve})$. By transversality of $\wcalD_\alpha(\xi_{s_0+\ve})$ and $\wcalA_\delta(\xi_{s_0+\ve})$, the other side of the half-disk intersects $\wcalA_\delta(\xi_{s_0+\ve})$ with the opposite orientation in a fiber close to $L$. This implies that there is a unique AL-graph $I'\in \acalM_{\Gamma(\vec{k})}(\Sigma;\xi_{s_0+\ve},\xi_2,\ldots,\xi_{3n})$ that is close to $I$ and whose first edge is nearly parallel to that of $I$. Moreover, the signs of $I$ and $I'$ are opposite. Hence, no terms of the AL-graphs whose first edge visits $\alpha$ survive in the sum $Z_n(\xi_{s_0+\ve},\xi_2,\ldots,\xi_{3n})$. The equality for other graphs can be proved similarly as Lemma~\ref{lem:case(a)}.

When $H_0$ does not have a trivalent vertex, then $H_0$ is a single flow-line between a pair of critical loci, say, from $\alpha$ to $\alpha'$. In this case, replace $\wcalA_\delta(\xi_{s_0+\ve})$ with $\wcalA_{\alpha'}(\xi_{s_0+\ve})$ in the previous paragraph so that similar cancellation occurs.
\end{proof}

\begin{proof}[Proof of Theorem~\ref{thm:main}]
The theorem follows as a corollary of Lemmas~\ref{lem:generic-ho}, \ref{lem:nobifurcations}, \ref{lem:case(a)}, \ref{lem:case(b)}, \ref{lem:case(d)}, \ref{lem:case(e)}, \ref{lem:case(f)} and \ref{lem:bd}. 
\end{proof}

%%%%%%%%%%%%%%%%%%%%%%%%%%%%%%
%%%%%%%%%%%%%%%%%%%%%%%%%%%%%%
\mysection{Surgery formulas of $\wZ_n$ and $\calQ$}{s:sformula}

%%%%%%%%%%%%%%%%%%%%%%%%%%%%%%
\subsection{Torelli surgery with trivial action of AL-paths}\label{ss:AL}

Let $I=[\alpha,\beta]\subset S^1$ be a small interval which does not contain any integers. Let $\Sigma'=\kappa^{-1}(\beta)$, $\Sigma''=\kappa^{-1}(\alpha)$ and $M_I=\kappa^{-1}(I)$. For an oriented fiberwise Morse function $f:M\to \R$ and its fiberwise gradient $\xi$, let $\omega$ be the restriction of $\xi$ on $M_I$. Suppose that there are no $1/1$-intersections in $M_I$ for $\omega$. For simplicity, we assume Assumption~\ref{assum:mu0} for a spin structure $\mathfrak{s}$ on $M$. Choose the canonical stable framing $\phi$ as in Lemma~\ref{lem:cframing} that is compatible with $\mathfrak{s}$. Let $B'\subset \Sigma'$ be a small disk that is disjoint from critical loci of $\xi$ and put $\Sigma_\circ'=\Sigma'\setminus \mathrm{Int}\,B'$. The negative gradient of $\kappa$ induces a diffeomorphism $g_0:\Sigma'\to \Sigma''$. The orbit of $B'$ under the flow of $-\mathrm{grad}\,\kappa$ forms a cylinder $N\cong D^2\times I$ in $M_I$. Put $B''=\Sigma''\cap N$ and $\Sigma_\circ''=\Sigma''\setminus \mathrm{Int}\,B''$. 

Now we shall define a surgery of $(M,\xi)$ along $M_I$. Take an orientation preserving diffeomorphism $g:\Sigma'\to \Sigma''$ such that $g|_{B'}=g_0|_{B'}$, an oriented fiberwise Morse function $H:\Sigma'\times I\to \R$ and its fiberwise gradient $\Omega$. Let $\bar{g}:\Sigma'\times[\alpha,\alpha+\eta]\to \kappa^{-1}[\alpha,\alpha+\eta]$ ($\eta>0$ small) be a local trivialization of $\kappa$ that extends $g$, let $\bar{h}:\Sigma'\times [\beta-\eta,\beta]\to \kappa^{-1}[\beta-\eta,\beta]$ be a local trivialization of $\kappa$ that extends $\mathrm{id}$, and let $\bar{\nu}:B'\times I\to N$ be the local trivialization of the disk bundle $\kappa|_{N}$ such that $\bar{\nu}(z,t)$ is the intersection of the gradient line of $-\mathrm{grad}\,\kappa$ from $z\in B'$ with $\kappa^{-1}(t)$ and such that $\bar{\nu}|_{B'\times [\alpha,\alpha+\eta]}=\bar{g}|_{B'\times [\alpha,\alpha+\eta]}$ and $\bar{\nu}|_{B'\times [\beta-\eta,\beta]}=\bar{h}|_{B'\times [\beta-\eta,\beta]}$. Let $P^+(\Sigma',f)=\{p_{1,\Sigma'}^+,p_{2,\Sigma'}^+,\ldots,p_{r,\Sigma'}^+\}$ be the set of critical points of $f|_{\Sigma'}:\Sigma'\to \R$ and let $P^-(\Sigma'',f)=\{p_{1,\Sigma''}^-,p_{2,\Sigma''}^-,\ldots,p_{r,\Sigma''}^-\}$ be the set of critical points of $f|_{\Sigma''}:\Sigma''\to \R$ such that $p_{i,\Sigma'}^+$ and $p_{i,\Sigma''}^-$ are the endpoints of a critical locus of $f|_{M_I}$ for each $i$. Let $q_{i,\Sigma'}^-=g^{-1}(p_{i,\Sigma''}^-)\in \Sigma'\times\{\alpha\}$. 
\begin{Def}\label{def:adapted}
We say that $(\Sigma'\times I,B',g,H,\Omega)$ is {\it adapted to $(\kappa,f,\xi)$} if the following conditions (1)--(4) are satisfied.
\begin{enumerate}
\item On a collar neighborhood of $\Sigma'_\circ\times\{\alpha\}$ in $\Sigma'_\circ\times I$, $H$ and $\Omega$ agree with $\bar{g}^*f$ and $(d\bar{g}^{-1})\, \omega$ respectively.
\item On a collar neighborhood of $\Sigma'_\circ\times\{\beta\}$ in $\Sigma'_\circ\times I$, $H$ and $\Omega$ agree with $\bar{h}^*f$ and $(d\bar{h}^{-1})\,\omega$ respectively.
\item On $B'\times I$, $H$ and $\Omega$ agree with $\bar{\nu}^*f$ and $(d\bar{\nu}^{-1})\,\omega$ respectively.
\item The number of AL-paths of $\Omega$ from $p_{i,\Sigma'}^+$ to $q_{j,\Sigma'}^-$ counted with signs is the Kronecker delta $\delta_{ij}$.
\end{enumerate}
\end{Def}
The condition (4) is an analogue of pure braid with zero linking numbers. According to \cite[Lemma~{4.10}]{Wa2}, the action of AL-paths agrees with the action of cobordism on homology. This together with the condition (4) implies that the diffeomorphism $g^{-1}\circ g_0:\Sigma'\to \Sigma'$ represents an element of the Torelli group.
\begin{Def}
For $\tau=(\Sigma'\times I,B',g,H,\Omega)$ that is adapted to $(\kappa,f,\xi)$, we define the {\it adapted surgery} $M_\tau$ as the 3-manifold obtained from $M$ by removing $\mathrm{Int}(M_I\setminus N)$ and pasting $\Sigma'_\circ\times I$ back by the diffeomorphism 
\[ \mathrm{id}\cup \bar{\nu}\cup g:\Sigma'_\circ\times\{\beta\}\cup (\partial B'\times I)\cup  \Sigma'_\circ \times\{\alpha\}\to \partial \overline{M_I\setminus N}. \]
\end{Def}

By the conditions (1), (2) and (3) for the adaptedness, $f$ and $\xi$ on $M\setminus \mathrm{Int}\,(M_I\setminus N)$ extends smoothly over $M_\tau$ by $H$ and $\Omega$ on $\Sigma'_\circ\times I$. We denote by $f_\tau$ and $\xi_\tau$ the resulting function and vector field on $M_\tau$ respectively. Let $\kappa_\tau:M_\tau\to S^1$ be the projection obtained by gluing $\kappa|_{M\setminus\mathrm{Int}\,(M_I\setminus N)}$ and $\mathrm{proj}:\Sigma'_\circ\times I\to I$. We say that $(\kappa_\tau,f_\tau,\xi_\tau)$ is obtained from $(\kappa,f,\xi)$ by adapted surgery with respect to $\tau$. 

\begin{Exa}[Borromean surgery, $C_2$-move or $Y$-surgery (\cite{Mat})]\label{exa:borromean}
Suppose that $\Sigma'\subset M$ is of genus 3 and the restriction of $f$ on $\Sigma'$ is a minimal Morse function, i.e., the numbers of critical points of index $0,1,2$ are $1,6,1$ respectively. Suppose that there is a level curve $L=(f|_{\Sigma'})^{-1}(h)$ of $f|_{\Sigma'}$ such that 
\begin{itemize}
\item three of the 6 critical points of $f|_{\Sigma'}$ of index 1 lie below $L$ and the rest of those of index 1 lie above $L$ and
\item $B=(f|_{\Sigma'})^{-1}(-\infty,h]$ has four boundary components, i.e., a disk with three holes.
\end{itemize} 
Then a Borromean 3-strand braid $[\sigma_1^2,\sigma_2^2]=(\sigma_1\sigma_2^{-1})^3$ gives rise to a relative diffeomorphism of $B$ relative to $\partial B$, which can be realized by a sequence of 1-handle-slides below $L$. One has a 1-parameter family of Morse functions on $B$ that induces the handle-slides, where the handle structures for the endpoints of the path may be assumed to be equal. Extending such 1-parameter family by trivial family above $L$, one obtains a 1-parameter family of Morse functions on $\Sigma'$ that gives rise to an adapted 5-tuple $(\Sigma'\times I,B',g,H,\Omega)$. It is known that such a surgery may yield topologically different 3-manifold. \qed
\end{Exa}

Suppose that a triple $(\kappa_1,f_1,\xi_1)$ is obtained from $(\kappa,f,\xi)$ by a small perturbation. Then we perturb $(\mathrm{proj}:\Sigma'_\circ\times I\to \R,H,\Omega)$ as follows. Put 
\[ V=\Sigma'_\circ\times([\alpha,\alpha+\eta]\tcoprod [\beta-\eta,\beta])\cup (B'\times I).\]
We consider the pullback $(\bar{h}\cup \bar{\nu}\cup \bar{g})^*\kappa_1:V\to \R$ where $\bar{h},\bar{\nu},\bar{g}$ are the maps fixed above by using $(\kappa,f,\xi)$. Then we deform $\mathrm{proj}:\Sigma'_\circ\times I\to I\subset \R$ slightly to $\Pi_1:\Sigma'_\circ\times I\to \R$ so that $\Pi_1$ agrees with $(\bar{h}\cup \bar{\nu}\cup \bar{g})^*\kappa_1$ on $V$. Similarly, we deform $H$ and $\Omega$ slightly to $H_1:\Sigma'_\circ\times I\to \R$ and $\Omega_1$ so that they agree with the pullbacks $(\bar{h}\cup\bar{\nu}\cup \bar{g})^*f_1$ and $(d\bar{h}^{-1}\cup d\bar{\nu}^{-1}\cup d\bar{g}^{-1})\xi_1$ on $V$. After surgery by $\tau$ with $(\Pi_1,H_1,\Omega_1)$, we get a triple $(\kappa_{1\tau},f_{1\tau},\xi_{1\tau})$ on $M_\tau$. We say that such a triple $(\kappa_{1\tau},f_{1\tau},\xi_{1\tau})$ is obtained by an {\it adapted perturbation} from $(\kappa_\tau,f_\tau,\xi_\tau)$ with respect to $(\kappa_1,f_1,\xi_1)$.

%%%%%%%%%%%%%%%%%%%%%%%%%%%%%%
\subsection{Spin structure and surgery}\label{ss:***}

Given a spin structure $\mathfrak{s}$ on $M$, one can choose a spin structure on $M_\tau$ as follows. As in \cite{Mas}, we consider a spin structure on an $n$-manifold $X$ as an element $\mathfrak{s}\in H^1(P;\Z_2)$, where $P\to X$ is the orthonormal frame bundle of $TX$, such that the restriction of $\mathfrak{s}$ to each fiber is non trivial in $H^1(SO(n);\Z_2)$. For a manifold $X$, let $\Spin(X)$ denote the set of spin structures on $X$. 

\begin{Lem}[\cite{Mas}]
Let $X$ be a closed oriented manifold obtained from two compact oriented spinnable manifolds $X_1$ and $X_2$ with connected boundaries $S_1,S_2$ respectively, by identifying the boundaries by an orientation reversing diffeomorphism $\varphi:S_2\to S_1$. Suppose that the set
\[ J=\{(\mathfrak{s}_1,\mathfrak{s}_2)\in\Spin(X_1)\times \Spin(X_2);\,\varphi^*(-\mathfrak{s}_1|_{S_1})=\mathfrak{s}_2|_{S_2}\} \]
is not empty. Then $X$ is spinnable and the restriction map
\[ r:\Spin(X)\to \Spin(X_1)\times \Spin(X_2) \]
is injective. (In such case, for each $(\mathfrak{s}_1,\mathfrak{s}_2)\in J$, there is a unique spin structure $r^{-1}(\mathfrak{s}_1,\mathfrak{s}_2)$ on $X$ that restricts to the given pair.)
\end{Lem}

For example, if $\tau=(\Sigma'\times I,B',g,H,\Omega)$ is adapted to $(\kappa,f,\xi)$, then in particular $g^{-1}\circ g_0$ induces the identity on $H_1(\Sigma';\Z)$, as guaranteed by the condition (4) for the adaptedness, and $g^{-1}\circ g_0$ acts trivially on the spin structure on $\Sigma'$. Hence the restriction of $\mathfrak{s}$ to $M\setminus \mathrm{Int}(M_I\setminus N)$ can be extended to a spin structure $\mathfrak{s}_\tau$ on $M_\tau$. We fix one such for each $\tau$. Note that by \cite[Lemma~5]{Mas}, one has $(\mathfrak{s}_\tau)_{\tau'}=(\mathfrak{s}_{\tau'})_\tau$ for disjoint surgeries $\tau, \tau'$.

%%%%%%%%%%%%%%%%%%%%%%%%%%%%%%
\subsection{Alternating sum of surgeries and filtrations}\label{ss:alternatingsum}

Let $I_1,I_2,\ldots, I_m\subset S^1$ be a disjoint collection of small intervals that are disjoint from $0\in S^1$. Let $\tau_1,\tau_2,\ldots,\tau_m$ be a sequence of 5-tuples on $M_{I_1},M_{I_2},\ldots,M_{I_m}$ as in Definition~\ref{def:adapted} that are adapted to $(\kappa,f,\xi)$. Suppose that there are no $1/1$-intersections in $M_{I_i}$ for $\xi$ for all $i$. Let $T=\{\tau_1,\tau_2,\ldots,\tau_m\}$ and put
\[ [M,T]
=\sum_{S\subset T}(-1)^{|S|}(M_S,[\kappa_S],[f_S]) \]
where $O_S$ denotes $(\cdots ((O_{\tau_1})_{\tau_2})\cdots)_{\tau_m}$ for each object $O$. We define
\[ \begin{split}
  Z_n([M,T])
&=\sum_{S\subset T}(-1)^{|S|}Z_n(M_S,[\kappa_S],[f_S]),\\
  \wZ_n([M,T],\mathfrak{s}_T)
&=\sum_{S\subset T}(-1)^{|S|}\wZ_n(M_S,\mathfrak{s}_S,[\kappa_S],[f_S]).
\end{split} \]
Let $\wcalF_n(M,\kappa)$ be the vector space over $\Q$ spanned by alternating sums $[M,T]$ with $|T|=n$ where $T$ is as in \S\ref{ss:alternatingsum}. Then the sequence $\{\wcalF_n(M,\kappa)\}_{n\geq 0}$ forms a descending filtration, i.e., $\wcalF_n(M,\kappa)\supset \wcalF_{n+1}(M,\kappa)$. Let $\beta[M,T]$ denote the alternating sum of $(M_S,[\kappa_S])$ that is obtained from $[M,T]$ by forgetting the classes of fiberwise Morse functions. Let $\calF_n(M,\kappa)$ be the vector space over $\Q$ spanned by alternating sums $\beta[M,T]$ with $|T|=n$ where $T$ is as in \S\ref{ss:alternatingsum}. Again, $\{\calF_n(M,\kappa)\}_{n\geq 0}$ forms a descending filtration. We consider the following problem.
\begin{Prob} For each $n\geq 1$, determine the structure of the quotient space
\begin{enumerate}
\item $\wcalF_n(M,\kappa)/\wcalF_{n+1}(M,\kappa)$
\item $\calF_n(M,\kappa)/\calF_{n+1}(M,\kappa)$
\end{enumerate}
\end{Prob}
Although we do not have a solution to the problem, the invariants $\wZ_n$ and $\calQ$ are helpful for understanding of the quotient spaces. We do not know whether the filtrations are the right ones that correspond to the space $\calA_n(\Lambda)$, but we think that there should be a close connection of the filtrations to $\calA_n(\Lambda)$. 

%\clearpage

%%%%%%%%%%%%%%%%%%%%%%%%%%%%%%
\subsection{Surgery formula for adapted surgery}\label{ss:***}

Let $\kappa_i:M\to S^1$, $f_i:M\to \R$, $\xi_i$ ($i=1,2,\ldots,3n$) be as in \S\ref{ss:moduli_hori}. Suppose that for each $i$ the triple $(\kappa_i,f_i,\xi_i)$ is close to $(\kappa,f,\xi)$ with respect to the $C^\infty$-topology so close that they are obtained from $(\kappa,f,\xi)$ by a small isotopy in $M$. Let $\tau_1,\ldots,\tau_m$ be as above that are adapted to $(\kappa,f,\xi)$ and let $T=\{\tau_1,\tau_2,\ldots,\tau_m\}$. Let $(\kappa_{i\tau_j},f_{i\tau_j},\xi_{i\tau_j})$, $i=1,2,\ldots,3n$, $j=1,2,\ldots,m$, be triples obtained by adapted perturbations from $(\kappa_{\tau_j},f_{\tau_j},\xi_{\tau_j})$ with respect to $(\kappa_i,f_i,\xi_i)$. Moreover, replacing $\tau_j$ with $S\subset T$, we obtain triples $(\kappa_{iS},f_{iS},\xi_{iS})$. 

Put $I_j=[\alpha_j\beta_j]$, $\Sigma_j'=\kappa_1^{-1}(\beta_j)$ and $\Sigma_j''=\kappa_1^{-1}(\alpha_j)$. Let $c_\ell$ be the critical locus of $f$ that intersects $\Sigma'_j$ and $\Sigma''_j$ at $p_{\ell,\Sigma'_j}^+$ and $p_{\ell,\Sigma''_j}^-$ respectively. The isotopy which takes $(\kappa,f,\xi)$ to $(\kappa_i,f_i,\xi_i)$ also takes critical loci $c_\ell$ of $f$ to those of $f_i$. Let $c_\ell'$ be the critical locus of $f_i$ corresponding to $c_\ell$ by the isotopy. Then let $P^+(\Sigma_j',f_i)=\Sigma_j'\cap \{c_1',c_2',\ldots,c_r'\}$ and $P^-(\Sigma_j'',f_i)=\Sigma_j''\cap \{c_1',c_2',\ldots,c_r'\}$. 

For a fixed number $i\in\{1,2,\ldots,m\}$ and for a triple $j,k,\ell\in \{1,2,\ldots,3n\}$, take critical points $x_j\in P^+(\Sigma_i',f_j)\cup P^-(\Sigma_i'',f_j)$, $x_k\in P^+(\Sigma_i',f_k)\cup P^-(\Sigma_i'',f_k)$, $x_\ell\in P^+(\Sigma_i',f_\ell)\cup P^-(\Sigma_i'',f_\ell)$ of index 1 in $M$. Then we define the {\it $Y$-graph} $Y(x_j,x_k,x_\ell)$ as the $Y$-shaped unitrivalent graph such that the three univalent vertices are labeled $x_j,x_k,x_\ell$ respectively. The order of $(x_j,x_k,x_\ell)$ determines a vertex-orientation of the trivalent vertex of $Y(x_j,x_k,x_\ell)$. We impose the relation $Y(\sigma(a),\sigma(b),\sigma(c))=\mathrm{sgn}(\sigma)\,Y(a,b,c)$ for $\sigma\in\mathfrak{S}_3$. Let $\acalM_{Y(x_j,x_k,x_\ell)}(\xi_{j\tau_i},\xi_{k\tau_i},\xi_{\ell\tau_i})$ be the moduli space of AL-graphs in $\Sigma_{i\circ}'\times I_i\subset M_{\tau_i}$ for $Y(x_j,x_k,x_\ell)$. The coorientation of the moduli space is determined as the wedge product of the coorientations of the loci of the descending and the ascending manifolds of critical loci considered at the trivalent vertex. We take the wedge product of the coorientations using the vertex-orientation. Put
\[ 
  Z_{n,Y}^T(i)=\sum_{j,k,\ell}\,\asum{x_j,x_k,x_\ell}{\mathrm{ind}=1}\#\acalM_{Y(x_j,x_k,x_\ell)}(\xi_{j\tau_i},\xi_{k\tau_i},\xi_{\ell\tau_i}) \, Y(x_j,x_k,x_\ell).
\]
Note that the product $\#\acalM_{Y(x_j,x_k,x_\ell)}(\xi_{j\tau_i},\xi_{k\tau_i},\xi_{\ell\tau_i}) \, Y(x_j,x_k,x_\ell)$ does not depend on the choice of vertex-orientation. We define $Z_{n,Y}^T(i)_0$ by the same formula as $Z_{n,Y}^T(i)$ but restricting the sum $\sum_{x_j,x_k,x_\ell}$ to triples $(x_j,x_k,x_\ell)$ consisting only of points of $P^+$ or only of points of $P^-$.

For a pair $i,j\in\{1,2,\ldots,m\}$ and a number $k\in \{1,2,\ldots,3n\}$, take critical points $x\in P^+(\Sigma_i',f_k)\cup P^-(\Sigma_i'',f_k)$ and $y\in P^+(\Sigma_j',f_k)\cup P^-(\Sigma_j'',f_k)$ in $\Sigma_{i\circ}'\times I_i$ and $\Sigma_{j\circ}'\times I_j$ respectively of index 1. Suppose that $x,y$ are such that if $x\in P^+$ then $y\in P^{-}$ and if $x\in P^-$ then $y\in P^+$. Then we define a {\it chord} $C(x,y)$ as the connected univalent graph ${}_x\kern-1mm\curvearrowright_y$ with vertices labeled by $x$ and $y$ respectively. Let $\acalM_{C(x,y)}(\xi_{k\tau_i},\xi_{k\tau_j})$ be the moduli space of AL-paths in $M_{\{\tau_i,\tau_j\}}$ going from $x$ to $y$. This gives a 0-dimensional chain in $\bConf_{K_2}(M)$, whose twisted homology class is represented by a rational function in $\widehat{\Lambda}$. Put
\[ Z_C^T(k)=\sum_{i,j}\sum_{x,y}\#\acalM_{C(x,y)}(\xi_{k\tau_i},\xi_{k\tau_j})\,C(x,y), \]
where $\#\acalM_{C(x,y)}(\xi_{k\tau_i},\xi_{k\tau_j})\in\widehat{\Lambda}$.

We define a product of several $Y$-graphs by the disjoint union and extend it to the $\Q$-linear combinations of products of $Y$-graphs by $\Q$-multilinearity. We also define a product of several $\widehat{\Lambda}$-colored chords by the disjoint union and extend it by $\Q$-multilinearity. Then we define a non-symmetric pairing $\langle A,B\rangle\in\calA_n(\widehat{\Lambda})$, where $A$ is a product of $2n$ $Y$-graphs and $B$ is a product of $3n$ $\widehat{\Lambda}$-colored chords, as follows. We set $\langle A,B\rangle=0$ unless the labels of the $6n$ legs in $A$ and $B$ match, i.e., unless the sets of the labels of the $6n$ univalent vertices in $A$ and $B$ agree. If the labels of the $6n$ legs in $A$ and $B$ match, we set $\langle A,B\rangle$ to be the $\widehat{\Lambda}$-colored graph obtained by joining all pairs of legs which have the same label. We extend the pairing to $\Q$-linear combinations by $\Q$-linearity. For example,
\[ \left\langle \fig{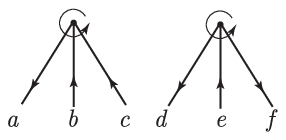}\,\,\,,\,\,\,\, \fig{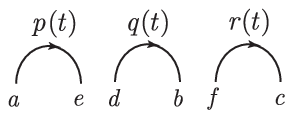}\right\rangle
=\left[\fig{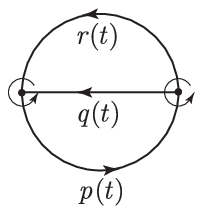}\right].\]

\begin{Thm}\label{thm:surgery}
Let $T$ be as in \S\ref{ss:alternatingsum} and let $m=|T|$. 
\begin{enumerate}
\item If $m>2n$, then 
\[ \wZ_n([M,T],\mathfrak{s}_T)=0. \]

\item If $m=2n$, then 
\[ \wZ_n([M,T],\mathfrak{s}_T)
=(2n)!\,\Bigl\langle \prod_{i=1}^{2n}Z_{n,Y}^T(i),\prod_{k=1}^{3n}Z_C^T(k)\Bigr\rangle
\in\calA_n(\widehat{\Lambda}). \]
\end{enumerate}
\end{Thm}
Theorem~\ref{thm:surgery} shows that the values of $\wZ_n$ for $[M,T]$ with $|T|\geq 2n$ does not depend on the choice of spin structures $\mathfrak{s}_S$ for $S\subset T$, hence $\wZ_n$ induces a well-defined linear map
\[ \wZ_n:\wcalF_n(M,\kappa)/\wcalF_{n+1}(M,\kappa)\to \calA_n(\widehat{\Lambda}). \]
\begin{proof}[Proof of Theorem~\ref{thm:surgery}]
We shall only give an outline of the proof since similar argument has been given previously by several authors (e.g., \cite{AF,KT}) for different purposes. We assume that for each $i$ the triple $(\kappa_i,f_i,\xi_j)$ is close to $(\kappa,f,\xi)$ as above. For each $S\subset T$ and for each $i$ such that $\tau_i\in S$, we put $V_i=\Sigma_{i\circ}'\times I_i\subset M_S$. If $m\geq 2n$, an AL-graph $G$ in $M_S$, $S\subset T$, having a trivalent vertex in the complement of $\bigcup_{i:\tau_i\in S}V_i$ does not survive in the alternating sum $Z_n([M,T])$ since in such a case either of the following occurs. 
\begin{enumerate}
\item[(a)] There is $j$ such that no trivalent vertex of $G$ belongs to $V_j\subset M_S$.
\item[(b)] There is a unique graph $G'$ in $M_{S\cup\{\tau_j\}}$ that represents the same monomial $\Lambda$-colored graph as $G$, that agrees with $G$ outside $V_j=\Sigma_{j\circ}'\times I_j\subset M_{S\cup\{\tau_j\}}$ and that has no trivalent vertex in $V_j$.
\end{enumerate}
In the case (a), $G$ collapses to an AL-graph in $M_{S\setminus\{\tau_j\}}$ that represents the same monomial $\Lambda$-colored graph as $G$. By the condition (4) of Definition~\ref{def:adapted}, the counts of both graphs are equal. Such a pair cancels each other out in $Z_n(M_S,[\kappa_S],[f_S])-Z_n(M_{S\setminus\{\tau_j\}},[\kappa_S],[f_{S\setminus\{\tau_j\}}])$. In the case (b), similar cancellation occurs in \\$Z_n(M_{S\cup\{\tau_j\}},[\kappa_S],[f_{S\cup\{\tau_j\}}])-Z_n(M_S,[\kappa_S],[f_S])$. 

In the case $m>2n$, either of (a) or (b) always occurs and hence the claim (1) holds. In the case $m=2n$, the only surviving terms in the alternating sum $Z_n([M,T])$ are those for AL-graphs $G$ in $M_{T}$ whose set of $2n$ trivalent vertices intersects $V_j$ for every $j\in\{1,2,\ldots,2n\}$. This means that for each $j\in\{1,2,\ldots,2n\}$, there is a horizontal tree component of $G$ with one trivalent vertex, i.e., a $Y$-graph, in $V_j$. The sum of counts of such $G$ gives the right hand side of the formula of (2). The reason for the coefficient $(2n)!$ is that for each AL-graph $G$, the same contribution is counted for $(2n)!$ different labelings for the trivalent vertices of $G$. 

Finally, we should check that the alternating sum of the correction terms vanishes, if $m>1$. For simplicity, we assume Assumption~\ref{assum:mu0} for $(M,\mathfrak{s})$ and $(M_{\tau_i},\mathfrak{s}_{\tau_i})$ for all $i$. If $(M_{\tau_i},\mathfrak{s}_{\tau_i})$ does not satisfy Assumption~\ref{assum:mu0}, we have only to replace $\tau_i$ with the 16 iterations of $\tau_i$ when defining the correction term and to divide the result by 16. By Assumption~\ref{assum:mu0} for $(M,\mathfrak{s})$ and $(M_{\tau_i},\mathfrak{s}_{\tau_i})$, there exist spin cobordisms $W_i$ with $\mathrm{sign}\,W_i=0$ (mod 16) that spin bound $(-V_i)\cup \overline{M_{I_i}\setminus N_i}$. Then one can find sequences $\vec{\rho}_{W_i}$ of sections of $T^vW_i$ extending given ones on the boundaries, and the correction term in $\wZ_n(M_S,\mathfrak{s}_S,[\kappa_S],[f_S])$ is 
\[ -\sum_{\tau_j\in S}Z_n^\anomaly(\vec{\rho}_{W_j})-Z_n^\anomaly(\vec{\rho}_W). \]
The total alternating sum is as follows.
\[ \begin{split}
  &\sum_{S\subset T}(-1)^{|S|}\Bigl[-\sum_{\tau_j\in S}Z_n^\anomaly(\vec{\rho}_{W_j})-Z_n^\anomaly(\vec{\rho}_W)\Bigr]\\
 =&\sum_{j=1}^m -\Bigl(\#\{S;\,\tau_j\in S,\, |S|\mbox{ is even}\}-\#\{S;\,\tau_j\in S,\, |S|\mbox{ is odd}\}\Bigr)\,Z_n^\anomaly(\vec{\rho}_{W_j})\\
 =&\sum_{j=1}^m -(1-1)^{m-1}\,Z_n^\anomaly(\vec{\rho}_{W_j})=0.
\end{split}\]
\end{proof}

We have another surgery formula analogous to Theorem~\ref{thm:surgery} for $\calQ$. The following theorem is a special case of Lescop's Lagrangian preserving surgery formula of $\calQ$ in \cite{Les2}.
\begin{Thm}\label{thm:surgery2}
Let $T$ be as in \S\ref{ss:alternatingsum} and let $m=|T|$. 
\begin{enumerate}
\item If $m>2$, then 
\[ \calQ(\beta[M,T])=0. \]

\item If $m=2$, then 
\[ \calQ(\beta[M,T])
=2!\,\bigl\langle Z_{1,Y}^T(1)_0Z_{1,Y}^T(2)_0,Z_C^T(1)Z_C^T(2)Z_C^T(3)\bigr\rangle
\in\calA_1(\widehat{\Lambda})/O_\delta. \]
\end{enumerate}
\end{Thm}
The proof of Theorem~\ref{thm:surgery2} is almost parallel to Theorem~\ref{thm:surgery}. We remark that for $m\geq 2$, no contribution of $\partial\bConf_{K_2}(M)$ survive in the alternating sum. Hence so are the correction terms involving $U_{\xi_i}$. 

Theorem~\ref{thm:surgery2} (1) shows that $\calQ$ induces a well-defined linear map
\[ \calQ:\calF_2(M,\kappa)/\calF_3(M,\kappa)\to \calA_1(\widehat{\Lambda})/O_\delta \]
and (2) gives an explicit formula for this map. 

%%%%%%%%%%%%%%%%%%%%%%%%%%%%%%
%%%%%%%%%%%%%%%%%%%%%%%%%%%%%%
\mysection{AL-paths and the homology of $M$}{s:homology}

%%%%%%%%%%%%%%%%%%%%%%%%%%%%%%
\subsection{A chain complex via AL-paths}

Let $a=0$, $b=\frac{1}{2}\in S^1=\R/\Z$. Let $f:M\to \R$ be an oriented fiberwise Morse function for the fibration $\kappa$ and let $\xi$ be its gradient along the fibers. Suppose that the $1/1$-intersections in $\xi$ are disjoint from both $\kappa^{-1}(a)$ and $\kappa^{-1}(b)$. Put $M^+=\kappa^{-1}[0,\frac{1}{2}]$ and $M^-=\kappa^{-1}[\frac{1}{2},1]$. For $c\in S^1$, put $\Sigma_c=\kappa^{-1}(c)$, $f_c=f|_{\Sigma_c}$ and $\xi_c=\xi|_{\Sigma_c}$ and let $\Sigma(f_c)$ denote the set of critical points of $f_c$. Let $\Sigma_i(f_c)$ denote the set of critical points of $f_c$ of index $i$. 

Now we define a chain complex $C_*^\mathrm{AL}(\xi)$ using AL-paths. For $p\in \Sigma(f_a)$, let $\mathsf{D}_p^-(\xi)$ be the subset of $\bacalM_{K_2}(\xi)$ consisting of AL-paths from $p$ to a point of $M^-$ that does not intersect $\mathrm{Int}\,M^+=\kappa^{-1}(0,\frac{1}{2})$. As paths in $M^+$, we consider a {\it reverse} AL-path, which is defined by using $-\kappa$ instead of $\kappa$ in the definition of AL-path. Let $\bcalM_{K_2}^{\mathrm{rAL}}(\xi)$ be the (closed) moduli space of reverse AL-paths in $M$. For $p\in \Sigma(f_a)$, let $\mathsf{D}_p^+(\xi)$ be the subset of $\bcalM_{K_2}^{\mathrm{rAL}}(\xi)$ consisting of reverse AL-paths from $p$ to a point of $M^+$ that does not intersect $\mathrm{Int}\,M^-=\kappa^{-1}(\frac{1}{2},1)$. We orient $\mathsf{D}_p^+(\xi)$ and $\mathsf{D}_p^-(\xi)$ by using the orientations of descending and ascending manifold loci and the orientations of level surfaces determined by $-\mathrm{grad}\,\kappa$, as in \S\ref{ss:coori} or \cite[\S{3}]{Wa2}. Put
\[ \mathsf{D}_p(\xi)=\mathsf{D}_p^-(\xi)\cup \mathsf{D}_p^+(\xi). \]
The assignment of the terminal endpoint of a path defines a continuous map
\[ \bar{b}:\mathsf{D}_p(\xi)\to M. \]
For $q\in \Sigma(f_b)$, let $\calD_q(\xi_b)$ denote the descending manifold of $q$ for $-\xi_b$ and let $\bcalD_q(\xi_b)$ be its compactification to the space of possibly broken flow lines (see e.g., \cite{BH}). Let $\bar{\beta}:\bcalD_q(\xi_b)\to \Sigma_b$ denote the continuous map that extends the inclusion $\calD_q(\xi_b)\to \Sigma_b$. 

Roughly, the set of maps $\bar{b}$ and $\bar{\beta}$ gives a cell-like structure on $M$ whose degree $i$ part is generated by $\bar{b}(\mathsf{D}_p(\xi))$ for $p\in\Sigma_{i-1}(\xi_a)$ and $\bar{\beta}(\bcalD_q(\xi_b))$ for $q\in\Sigma_i(\xi_b)$. The incidence coefficients are determined as follows. For $p,q\in \Sigma(\xi_a)\cup \Sigma(\xi_b)$, let $\mathsf{M}_2(\xi;p,q)$ be the subspace of $\mathsf{D}_p(\xi)$ or $\bcalD_p(\xi_b)$ consisting of AL-paths (or reverse AL-paths) from $p$ to $q$. If $\mathrm{ind}\,p$ and $\mathrm{ind}\,q$ satisfies the conditions (1), (2), (3) below and if $\xi$ is generic, then $\mathsf{M}_2(\xi;p,q)$ is a compact 0-dimensional manifold. 
\begin{enumerate}
\item $p\in\Sigma(f_a)$, $q\in \Sigma(f_a)$ and $\mathrm{ind}\,p-\mathrm{ind}\,q=1$.
\item $p\in\Sigma(f_a)$, $q\in \Sigma(f_b)$ and $\mathrm{ind}\,p=\mathrm{ind}\,q$.
\item $p\in\Sigma(f_b)$, $q\in \Sigma(f_b)$ and $\mathrm{ind}\,p-\mathrm{ind}\,q=1$.
\end{enumerate}
Moreover, in such cases $\mathsf{M}_2(\xi;p,q)$ has a natural orientation (sign) as follows. If $p,q$ satisfies (1) or (3) above, then we orient $\mathsf{M}_2(\xi;p,q)$ as usual, namely, comparing $o^*_M(\calD_p(\xi_c))_x\wedge o^*_M(\calA_q(\xi_c))_x$ ($x\in \mathsf{M}_2(\xi;p,q)$, $c\in\{a,b\}$) with the orientation of the level curve $\ell_x\subset \Sigma_c$ of $f_c$ including $x$. Here, the orientation of $\ell_x$ is given by $o(\ell_x)_x=\iota(-\mathrm{grad}_xf_c)\,o(\Sigma_c)_x$. If $p,q$ satisfies (2) above, then we orient $\mathsf{M}_2(\xi;p,q)$ by the signs of AL-paths from $p$ to $q$. 

Now we put
\[ E_{0,j}=C_j(f_b),\quad E_{1,j}=C_j(f_a). \]
Let $\Phi_0(f_a):E_{1,j}\to E_{1,j-1}$, $\Phi_0(f_b):E_{0,j}\to E_{0,j-1}$ be the boundary operators of Morse complexes $C_*(f_a),C_*(f_b)$ respectively. Namely, 
\[ \begin{split}
  \Phi_0(f_a)(p) &=\sum_{q\in\Sigma_{j-1}(f_a)}\#\mathsf{M}_2(\xi;p,q)\, q \quad (p\in\Sigma_j(f_a))\\
  \Phi_0(f_b)(p') &=\sum_{q'\in\Sigma_{j-1}(f_b)}\#\mathsf{M}_2(\xi;p',q')\,q' \quad (p'\in\Sigma_j(f_b))\\
 \end{split}\]
Let $\Phi_1(f):E_{1,j}\to E_{0,j}$ be the operator defined by
\[ \Phi_1(f)(p) =\sum_{q\in\Sigma_j(f_b)}\#\mathsf{M}_2(\xi;p,q)\,q\quad (p\in \Sigma_j(f_a)).\]
Put $C_j^\mathrm{AL}(\xi)=E_{0,j}\oplus E_{1,j-1}$ and define $\partial^\mathrm{AL}:C_j^\mathrm{AL}(\xi)\to C_{j-1}^\mathrm{AL}(\xi)$ by
\[ \partial^\mathrm{AL} p = \left\{\begin{array}{ll}
-\Phi_0(f_a)(p) + \Phi_1(f)(p) & \mbox{if $p\in\Sigma(f_a)$}\\
\Phi_0(f_b)(p) & \mbox{if $p\in \Sigma(f_b)$}
\end{array}\right.\]
This definition is motivated by the twisted tensor product in \cite{Ig2}. 
\begin{Prop}\label{prop:AL-homology}
The pair $(C_*^\mathrm{AL}(\xi),\partial^\mathrm{AL})$ forms a chain complex and its homology is canonically isomorphic to $H_*(M;\Z)$.
\end{Prop}
\begin{proof}
Since $\Phi_0(f_a)\Phi_0(f_a)=0$ and $\Phi_0(f_b)\Phi_0(f_b)=0$, we have 
\[ \partial^{\mathrm{AL}}\partial^{\mathrm{AL}}=\Phi_0(f_b)\Phi_1(f)-\Phi_1(f)\Phi_0(f_a).\]
Thus the condition $\partial^{\mathrm{AL}}\partial^{\mathrm{AL}}=0$ is equivalent to the condition that $\Phi_1(f)$ is a chain map. It is enough to see that $\Phi_0(f_b)\Phi_1(f)-\Phi_1(f)\Phi_0(f_a)$ is given by the boundaries of the 1-dimensional moduli spaces $\mathsf{M}_2(\xi;p,q)$ for $p\in\Sigma_i(f_a)$, $q\in\Sigma_{i-1}(f_b)$. From the definition of the moduli space of AL-paths given in \cite{Wa2}, it follows that $\partial \mathsf{M}_2(\xi;p,q)$ consists of AL-paths between $p$ and $q$ of the forms $\gamma_1*\gamma_2$, where $\gamma_1$ goes from $p$ to some $r\in\Sigma_{i-1}(f_a)\cup\Sigma_i(f_b)$ and $\gamma_2$ goes from $r$ to $q$. This implies that 
\[ 0=\#\partial\mathsf{M}_2(\xi;p,q)=\sum_{r\in \Sigma_{i-1}(f_a)\cup\Sigma_i(f_b)}\pm \#\mathsf{M}_2(\xi;p,r)\times\#\mathsf{M}_2(\xi;r,q). \]
It suffices to check that this differs by a multiple of $\pm 1$ from the coefficient of $q$ in $(\Phi_0(f_b)\Phi_1(f)-\Phi_1(f)\Phi_0(f_a))(p)$. We assume for simplicity that the bundle $\kappa$ is trivial and that $f$ is a trivial 1-parameter family of a Morse function on $\Sigma$ since the mechanism of inducing boundary orientations for general case is the same except that the signs of AL-paths with $1/1$-intersections are multiplied. Moreover, we have only to check the claim for a trivial family over a geometric 1-simplex $|\Delta^1|=[0,1]$ (replacing $a,b$ with $a=0$, $b=1$ as elements of $[0,1]$ respectively). Since the orientations of $\Sigma_a$ and $\Sigma_b$ induced from that of $\Sigma\times |\Delta^1|$ are opposite, the orientation of $\mathsf{M}_2(\xi;p,q)$ induces opposite signs in the boundary operators on $\Sigma_a$ and $\Sigma_b$. Hence the coefficient of $q$ in 
\[ (\Phi_0(f_b)\Phi_1(f)-\Phi_1(f)\Phi_0(f_a))(p)=\partial^\mathrm{AL}\partial^\mathrm{AL}p \]
is $\pm \#\partial\mathsf{M}_2(\xi;p,q)=0$.

Since $\Phi_1(f)$ is a chain map, the system $(\{C_*(f_a),C_*(f_b)\},\Phi_0,\Phi_1,0,0,\ldots)$ satisfies the axiom of Igusa's $A_\infty$-functor (\cite{Ig2}). As mentioned above, $(C_*^\mathrm{AL},\partial^\mathrm{AL})$ agrees with Igusa's twisted complex for the $A_\infty$-functor, which gives the homology of $M$. (One may use Mayer--Vietoris argument and \cite[Lemma~4.10]{Wa2} to see this directly.)
\end{proof}

\begin{Cor}\label{cor:ZYnontrivial}
Suppose that the fiberwise Morse function $f$ is such that $f_a$ is minimal, i.e., $\#\Sigma_0(f_a)=1$ and $\#\Sigma_2(f_a)=1$. Suppose moreover that $\Phi_1(f)=0$. If $(\kappa_i,f_i,\xi_i)$ are sufficiently close to $(\kappa,f,\xi)$ and generic, then the element
\[ \sum_{p_1,p_2,p_3}\#\mathsf{D}_{p_1}(\xi_1)\cap\mathsf{D}_{p_2}(\xi_2)\cap\mathsf{D}_{p_3}(\xi_3)\,[p_1]\wedge [p_2]\wedge [p_3]\in \textstyle\bigwedge^3 H_2(M;\Q) \]
where the sum is over $p_1\in\Sigma_1((f_{1})_a)$, $p_2\in\Sigma_1((f_{2})_a)$, $p_3\in\Sigma_1((f_{3})_a)$, is well-defined and is a multiple of the dual to the triple cup product on $H^1(M;\Q)$. Hence $Z_{n,Y}^T(i)$ (or $Z_{n,Y}^T(i)_0$) is nontrivial if the triple cup product $\bigwedge^3 H^1(M;\Q)\to H^3(M;\Q)=\Q$ of the mapping torus of the surgery $\tau_i$ is nontrivial.
\end{Cor}

\begin{Exa}
Let $\Sigma_{g,k}$ denote an oriented surface of genus $g$ with $k$ boundary circles and let $\varphi_0:\Sigma_{1,1}\to \Sigma_{1,1}$ be a diffeomorphism such that $\varphi_0|_{\partial\Sigma_{1,1}}=\mathrm{id}$ and such that the induced map $\varphi_{0*}:H_1(\Sigma_{1,1};\Z)\to H_1(\Sigma_{1,1};\Z)$ is represented by the matrix $
A_0=\left(\begin{array}{cc}
  2 & 1\\
  1 & 1\\
\end{array}\right)=\left(\begin{array}{cc}
  1 & 1\\
  0 & 1\\
\end{array}\right)\left(\begin{array}{cc}
  1 & 0\\
  1 & 1\\
\end{array}\right)$. Since $A_0$ is a product of elementary matrices, $\varphi_0$ can be realized by 1-handle-slides in the standard handle presentation of $\Sigma_{1,1}$ in Figure~\ref{fig:spine-transverse} (b). By boundary connected sums between the 0-handles, we obtain a diffeomorphism $\varphi':\Sigma_{3,1}\to \Sigma_{3,1}$ such that the induced map on $H_1$ is represented by $A=A_0\oplus A_0\oplus A_0$. Let $\varphi:\Sigma_{3,0}\to \Sigma_{3,0}$ be an extension of $\varphi'$ by the identity. Let $M$ be the mapping torus of $\varphi$ and let $\kappa:M\to S^1$ be the projection. Since $\det(A-1)=-1$ is invertible in $\Z$, we have $H_1(M)=\Z$ (see e.g., \cite[Lemma~A.1]{Wa2}). Let $f:M\to \R$ be an oriented fiberwise Morse function that restricts to a minimal Morse function on a fiber, which exists. Then $f$ satisfies the assumption of Example~\ref{exa:borromean}. Let $I_1,I_2\subset S^1$ be small disjoint intervals and put $M_{I_j}=\kappa^{-1}(I_j)$. We identify $M_{I_j}$ with $\Sigma_{3,0}\times I_j$. Consider two disjoint $Y$-surgeries $T=\{\tau_1,\tau_2\}$ on $M_{I_1},M_{I_2}$ respectively as in Example~\ref{exa:borromean}, as follows. A $Y$-surgery can be considered as a replacement of $H_j\times I_j$ where $H_j$ is a (thickened) ribbon graph on $\Sigma_{3,0}$ of genus 0 with four boundary components. Assume that the spines of $H_1$ and of $H_2$ intersect transversally as in Figure~\ref{fig:spine-transverse} (a). 
\begin{figure}
\fig{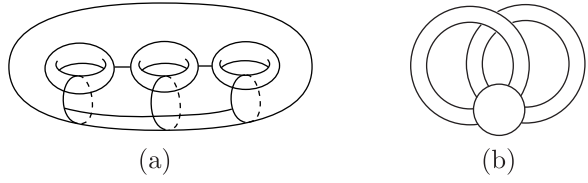}
\caption{(a) Spines of $H_1$ and $H_2$. (b) Standard handle presentation of $\Sigma_{1,1}$.}\label{fig:spine-transverse}
\end{figure}

From Corollary~\ref{cor:ZYnontrivial}, one sees that $Z_{1,Y}^T(1)_0$ and $Z_{1,Y}^T(2)_0$ are of the following forms:
\[ \begin{split}
	Z_{1,Y}^T(1)_0&=\delta\,Y(y_1',y_2',y_3')+\ve\,Y(x_1,x_2,x_3)+\mbox{(permutations)}\\
	Z_{1,Y}^T(2)_0&=\delta\,Y(w_1',w_2',w_3')+\ve\,Y(z_1,z_2,z_3)+\mbox{(permutations)}
\end{split}\]
where $\delta,\ve=\pm 1$ and $y_1',y_2',y_3',w_1',w_2',w_3'\in P^+$, $x_1,x_2,x_3,z_1,z_2,z_3\in P^-$ and ``permutations'' are the contributions of the corresponding triples of critical loci with the labels of the edges permuted. The coefficients in $Z_C^T(\ell)$ are given by matrix entries of 
\[ \begin{split}
  (1-tA)^{-1}&=\left(\begin{array}{cc}
    \frac{1-t}{d(t)} & \frac{t}{d(t)}\\
    \frac{t}{d(t)} & \frac{1-2t}{d(t)}
\end{array}\right)\oplus\left(\begin{array}{cc}
    \frac{1-t}{d(t)} & \frac{t}{d(t)}\\
    \frac{t}{d(t)} & \frac{1-2t}{d(t)}
\end{array}\right)\oplus\left(\begin{array}{cc}
    \frac{1-t}{d(t)} & \frac{t}{d(t)}\\
    \frac{t}{d(t)} & \frac{1-2t}{d(t)}
\end{array}\right)\\
tA(1-tA)^{-1}&=\left(\begin{array}{cc}
    \frac{t(2-t)}{d(t)} & \frac{t}{d(t)}\\
    \frac{t}{d(t)} & \frac{t(1-t)}{d(t)}
\end{array}\right)\oplus\left(\begin{array}{cc}
    \frac{t(2-t)}{d(t)} & \frac{t}{d(t)}\\
    \frac{t}{d(t)} & \frac{t(1-t)}{d(t)}
\end{array}\right)\oplus\left(\begin{array}{cc}
    \frac{t(2-t)}{d(t)} & \frac{t}{d(t)}\\
    \frac{t}{d(t)} & \frac{t(1-t)}{d(t)}
\end{array}\right)
\end{split}\]
where $d(t)=1-3t+t^2$. Hence for $T=\{\tau_1,\tau_2\}$, we have
\[ \begin{split}
	&\calQ(\beta[M,T])\\
	&=2!\,3!\,\delta\ve\,\Tr_\Theta\Bigl(\frac{1-t_1}{d(t_1)}\otimes\frac{1-t_2}{d(t_2)}\otimes\frac{1-t_3}{d(t_3)}+\frac{t_1(1-t_1)}{d(t_1)}\otimes\frac{t_2(1-t_2)}{d(t_2)}\otimes\frac{t_3(1-t_3)}{d(t_3)}\Bigr)\\
	&=24\,\delta\ve\,\Tr_\Theta\Bigl(\frac{1-t_1}{1-3t_1+t^2_1}\otimes\frac{1-t_2}{1-3t_2+t^2_2}\otimes\frac{1-t_3}{1-3t_3+t^2_3}\Bigr)
\end{split} \]
by Theorem~\ref{thm:surgery2}. In this way, one may find many nontrivial elements in $\calF_2/\calF_3$ and in $\wcalF_{2n}/\wcalF_{2n+1}$, where we must use $Z_{n,Y}^T(i)$ for $\wcalF_{2n}/\wcalF_{2n+1}$ in place of $Z_{1,Y}^T(i)_0$.
\qed
\end{Exa}

%%%%%%%%%%%%%%%%%%%%%%%%%%%%%%
%%%%%%%%%%%%%%%%%%%%%%%%%%%%%%
\mysection{Concluding remarks}{s:rem}

\subsection{Fatgraphs}
A {\it fatgraph} (or a ribbon graph) is a vertex-oriented graph. A fatgraph $G$ gives an oriented surface $F(G)$ with spine $G\subset F(G)$, which is a 2-dimensional handlebody consisting of 0- and 1-handles. A handle-slide in such a handlebody corresponds to a local move on a fatgraph. A sequence $\{G_i\to G_{i+1}\}_{i=0}^{k-1}$ of such local moves with $G_k=G_0$ gives rise to a (possibly unorientable) fiberwise Morse function for a surface bundle $M\to S^1$. 

It is known (e.g., \cite{Pe, MP, ABP, BKP}) that the mapping class group $\calM_{g,1}$ of punctured surface $\Sigma_{g,1}$ can be represented as the set of sequences of ``Whitehead moves'' $\{W_i:G_i\to G_{i+1}\}_{i=0}^{k-1}$ on (bordered) trivalent fatgraphs with $G_k=G_0$, modulo the relations called involutivity, commutativity and pentagon. A Whitehead move can be realized as a sequence of handle-slides. In this way, $\calM_{g,1}$ can be represented as the set of sequences of handle-slides on fatgraphs modulo the corresponding relations rewritten in terms of handle-slides. The relations can be realized by 2-parameter families of GMF's. It can be checked that the pentagon relation may not be realized by a concordance of oriented fiberwise Morse functions. Thus, to get an invariant of 3-manifolds, the method of this paper needs to be improved, although we do not know whether the pentagon relation changes the value of $\wZ_n$ or not.

%%%%%%%%%%%%%%%%%%%%%%%%%%%%%%
%%%%%%%%%%%%%%%%%%%%%%%%%%%%%%

\section*{\bf Acknowledgments.}
The author is supported by JSPS Grant-in-Aid for Young Scientists (B) 26800041.
\par

\end{document}